\documentclass[a4paper, reqno, 12pt]{amsart}

\usepackage{amsfonts, amsmath, amsthm}
\usepackage{latexsym}

\theoremstyle{plain}
\newtheorem{thm}{Theorem}[section]
\newtheorem{prp}[thm]{Proposition} 
\newtheorem{lem}[thm]{Lemma} 
\newtheorem{cor}[thm]{Corollary} 

\theoremstyle{definition}

\theoremstyle{remark}
\newtheorem{rmk}{Remark}[section]

\newtheorem{example}{Example}[section]

\numberwithin{equation}{section}

\newcommand{\N}{\mathbb{N}}

\newcommand{\R}{\mathbb{R}}
\newcommand{\C}{\mathbb{C}}
\newcommand{\Sph}{\mathbb{S}}
\newcommand{\calH}{\mathcal{H}_0}

\newcommand{\pa}{\partial}
\newcommand{\eps}{\varepsilon}
\newcommand{\jb}[1]{\langle #1 \rangle}
\newcommand{\dal}{\Box}

\newcommand{\Sum}{\sideset{}{'}\sum}  
\newcommand{\co}[1]{#1_{\bf c}}

\DeclareMathOperator{\realpart}{\rm Re}
\DeclareMathOperator{\imagpart}{\rm Im}
\DeclareMathOperator{\supp}{\rm supp}
\DeclareMathOperator{\diag}{\rm diag}

\title[Semilinear hyperbolic 
systems violating the null condition]
{Semilinear hyperbolic 
 systems violating the null condition}  

\author[S.~Katayama]{Soichiro Katayama}
\address{Department of Mathematics, Wakayama University,
             930 Sakaedani, Wakayama 640-8510, Japan.}
\email{katayama@center.wakayama-u.ac.jp}
\author[T.~Matoba]{Toshiaki Matoba}
\address{Osaka Prefectural Tennoji High School\\
   2-4-23 Sanmeicho, Abeno-ku, Osaka 545-0005, Japan.} 
\author[H.~Sunagawa]{Hideaki Sunagawa}
\address{Department of Mathematics, Graduate School of Science, 
             Osaka University, 
             1-1 Machikaneyama-cho, Toyonaka, Osaka 560-0043, Japan.} 
\email{sunagawa@math.sci.osaka-u.ac.jp}
 
\date{\today}   
\dedicatory{Dedicated to the memory of Professor Rentaro Agemi}

\keywords{Nonlinear wave equations; asymptotic behavior.}

\subjclass{Primary~35L71, Secondary~35B40.}

\begin{document}
\begin{abstract} 
We consider systems of semilinear wave equations in three space 
 dimensions with quadratic nonlinear terms not satisfying the null condition. 
We prove small data global existence of the classical solution 
 under a new structural condition related to the weak null condition. 
For two-component systems satisfying this condition, we also observe a new kind of asymptotic behavior:
Only one component is dissipated and the other one 
behaves like a free solution in the large time. 
\end{abstract}
\maketitle

\section{Introduction} 
This paper is concerned with global existence and large time behavior of
classical solutions to the Cauchy problem for systems of semilinear wave 
equations of the following type:
\begin{align}
& \Box u=F(\pa u) \qquad\qquad\qquad\qquad\qquad\quad\, \text{for $(t,x)\in (0,\infty)\times \R^3$},
\label{eq}\\
& u(0,x)=\eps f(x),\ (\pa_t u)(0,x)=\eps g(x) \quad \text{for $x\in \R^3$},
\label{data}
\end{align}
where $u=(u_1, \ldots, u_N)^{\rm T}$ is an $\R^N$-valued unknown function
of $(t,x)\in [0,\infty) \times \R^3$, 
$\Box=\pa_t^2-\Delta_x=\pa_t^2-\sum_{k=1}^3\pa_{k}^2$,
and $\pa u=(\pa_0 u, \pa_1 u, \pa_2 u, \pa_3 u)$ with the notation
$$
\pa_0=\pa_t=\frac{\pa}{\pa t},\ \pa_k=\frac{\pa}{\pa x_k},\ k=1,2,3.
$$
Here $B^{\rm T}$ stands for the transpose of a matrix (or vector) $B$.
For simplicity, we suppose that $f,g\in C^\infty_0(\R^3; \R^N)$,
and that the nonlinear term $F=(F_1,\ldots, F_N)^{\rm T}$ 
has the form
\begin{equation}
\label{Nonlinearity}
 F_j(\pa u)
 =
 \sum_{k,l=1}^N \sum_{a,b=0}^3 c_j^{kl,ab} (\pa_a u_k)(\pa_b u_l),
 \quad j=1,\ldots, N
\end{equation}
with some constants $c_j^{kl,ab}\in \R$. 
$\eps$ is a parameter which will be always
assumed to be sufficiently small.

Let us briefly review known results concerning the global existence and 
the asymptotic behavior.
In general, it is known that the solution to the Cauchy problem 
\eqref{eq}--\eqref{data} blows up in finite time no matter how small $\eps$ 
is; for example, 
if $N=1$ and $F(\pa u)=(\pa_t u)^2$, then the solution $u$ blows up in finite 
time for any $\eps>0$ unless $(f,g)\equiv(0,0)$ in \eqref{data} 
(see John~\cite{john}). 
Therefore, we need some restriction on the nonlinearity to 
obtain global solutions even for small data.
We say that we have {\it small data global existence} 
for the problem \eqref{eq}--\eqref{data} if for any $f,g\in C^\infty_0(\R^3)$,
there is a positive constant $\eps_0$ such that the Cauchy problem 
\eqref{eq}--\eqref{data} admits a global solution for any 
$\eps\in (0,\eps_0]$. 
Klainerman~\cite{kl} introduced a sufficient condition for 
small data global existence, known as the {\it null condition} (see also 
Christodoulou~\cite{ch}): 
We say that the null condition is satisfied if we have
\begin{equation}
 \label{NullCond}
 F^{\rm red}(\omega,Y)=0,\quad \omega=(\omega_1,\omega_2,\omega_3)\in \Sph^2,
 \ Y=(Y_1,\ldots, Y_N)^{\rm T}\in \R^N,
\end{equation}
where the {\it reduced nonlinearity} 
$F^{\rm red}(\omega,Y)
 =
 \bigl(F_1^{\rm red}(\omega,Y), \ldots, F_N^{\rm red}(\omega,Y)\bigr)^{\rm T}$ 
is defined by
\begin{equation}
\label{ReducedNonlinearity}
 F_j^{\rm red}(\omega, Y)
 :=
 F_j(\omega_0 Y, \omega_1 Y, \omega_2 Y, \omega_3 Y)
 =
 \sum_{k,l=1}^N\sum_{a,b=0}^3 c_j^{kl,ab} \omega_a\omega_b Y_kY_l
\end{equation}
for $\omega=(\omega_1,\omega_2,\omega_3)\in \Sph^2$ and 
$Y=(Y_1,\ldots, Y_N)^{\rm T}\in \R^N$ with $\omega_0=-1$. 
Here the constants $c_j^{kl,ab}$ are from \eqref{Nonlinearity}. 
In \cite{ch} and \cite{kl}, it was proved that the null condition implies 
small data global existence. It
is also easy to see that this global solution $u$ for small $\eps$ 
is {\it asymptotically free}, that is to say that
there is $(f^+,g^+)\in \dot{H}^1(\R^3; \R^N)\times L^2(\R^3;\R^N)$ such that
we have 
$$
 \lim_{t\to \infty} \|u(t)-u^+(t)\|_E=0
$$
for the solution $u^+$ to the free wave equation $\Box u^+=0$ with initial 
data $(u^+, \pa_t u^+)(0)=(f^+,g^+)$. Here and in the sequel, $\|\cdot\|_E$ is
the energy norm defined by 
$$
 \|\phi(t)\|_{E}= 
 \left(
 \frac{1}{2}\int_{\R^3} \left(|\pa_t \phi(t,x)|^2 + |\nabla_x \phi(t,x)|^2\right)dx
 \right)^{1/2},
$$
and $\dot{H}^1(\R^3)$ denotes the completion of $C^\infty_0(\R^3)$ 
with respect to the norm given by 
$\|\phi\|_{\dot{H}^1}=\|\nabla_x\phi\|_{L^2}$. 
Introducing the {\it null forms}
\begin{align}
\label{NF01}
Q_0(\phi,\psi)&:=(\pa_t\phi)(\pa_t\psi)-(\nabla_x\phi)\cdot(\nabla_x\psi),\\
\label{NF02}
Q_{ab}(\phi,\psi)&:=(\pa_a\phi)(\pa_b\psi)-(\pa_b\phi)(\pa_a\psi),
\ \ a,b \in \{0,1,2,3\},
\end{align}
we see that the nonlinearity $F$ of the form \eqref{Nonlinearity}
satisfies the null condition if and only if each component $F_j$ can be written
as a linear combination of the null forms $Q_0(u_k,u_l)$ and $Q_{ab}(u_k, u_l)$
with $k,l\in\{1,\ldots, N\}$ and $a, b\in\{0,\ldots, 3\}$. 

In connection with the Einstein equation
that can be expressed in wave coordinates as a system of quasilinear wave 
equations, Lindblad-Rodnianski~\cite{LinRod03} introduced the notion of 
the weak null condition, and proved the small data global existence for 
the Einstein equation in wave coordinates (see \cite{lr} and \cite{LinRod10}). 
The small data global existence is also obtained for a closely related equation
$$
 \dal u=\sum_{a,b=0}^3 g^{ab} u (\pa_a\pa_b u),
 \quad (t,x)\in [0,\infty)\times \R^3
$$
with constants $g^{ab}\in \R$, which satisfies the weak null condition, but 
violates the null condition (see Alinhac~\cite{Ali03}, and 
Lindblad~\cite{Lin92}, \cite{Lin08}). These successful examples suggest that 
the weak null condition implies the small data global existence in general; 
however this is still an open problem even for semilinear systems. 

Before we proceed to further discussion, we give the definition of 
the weak null condition for the semilinear case here: We say that the
weak null condition is satisfied if the reduced system
\begin{equation}
\label{ReducedSystem}
 \pa_t V(t ; \sigma, \omega)
 =
 -\frac{1}{2t}F^{\rm red}\bigl(\omega, V(t ; \sigma, \omega)\bigr),
 \quad t\ge 1,\ \sigma\in \R,\ \omega\in \Sph^2
\end{equation}
admits a global solution $V$ with at most polynomial growth of small power in 
$t$ for small data given at $t=1$, where \eqref{ReducedSystem} is obtained as 
an asymptotic equation for 
$V(t; \sigma, \omega)=(\pa_r-\pa_t)\bigl(ru(t, r\omega)\bigr)/2$ with 
$\sigma=r-t$ 
(see \eqref{ode_0} and \eqref{ode_V} below; see also H\"ormander~\cite{hor2}).
The null condition \eqref{NullCond} immediately implies the weak null 
condition. It is not easy to check whether or not the weak null 
condition is satisfied in general, because it depends on
the global behavior 
of the reduced system \eqref{ReducedSystem}.

In connection with the weak null condition,
Alinhac~\cite{al1} considered systems of semilinear wave equations and
introduced an algebraic condition to ensure the small data global existence.
His condition is stronger than the weak null condition, but still 
weaker than the null condition.
His condition was slightly extended, and the asymptotic behavior of 
global solutions under this extended condition was studied
in Katayama~\cite{ka} (see also Katayama-Kubo~\cite{kk}). 
Quite roughly speaking, the (extended) Alinhac condition says that the reduced 
system, through some change of unknowns, can be expressed as
\begin{equation}
\label{AliRe}
\pa_t V_j=
\begin{cases}
\displaystyle -\frac{1}{2t} \sum_{k=M+1}^N\sum_{l=1}^N C_{kl}(\omega) V_k V_l, 
& 1\le j \le M,\\
0,  
& M+1\le j\le  N
\end{cases}
\end{equation}
with smooth coefficients $C_{kl}(\omega)$ and $M\in \N$.
If the reduced system is of the form \eqref{AliRe}, then 
we can easily check that the weak null condition is satisfied;
however the null condition is violated unless all the coefficients 
$C_{kl}(\omega)$ vanish identically. 
Here we give three typical examples satisfying the (extended) Alinhac 
condition (and thus the weak null condition), 
but violating the null condition, and their asymptotic behavior 
(see \cite{ka} for details): 
The first example is
\begin{equation}
\label{FirstExampleA}
\begin{cases}
\dal u_1=(\pa_t u_2)(\pa_t u_1)+\text{(null forms)},\\
\dal u_2=\text{(null forms)}.
\end{cases}
\end{equation}
The second example is 
\begin{equation}
\label{SecondExampleA}
\begin{cases}
\dal u_1=(\pa_t u_2)^2+\text{(null forms)},\\
\dal u_2=\text{(null forms)}.
\end{cases}
\end{equation}
We can choose $f$ and $g$ in \eqref{data} such that 
we have $\|u(t)\|_E\ge C\eps (1+t)^{C\eps}$ for \eqref{FirstExampleA}, and 
$\|u(t)\|_E\ge C\eps\left(1+\eps\log(2+t)\right)$ for \eqref{SecondExampleA}, 
with a positive constant $C$. In both cases, the energy grows up to infinity, 
and the global solution $u$ is {\it not} asymptotically free for such data.
The third example is
\begin{equation}
\label{ThirdExampleA}
\begin{cases}
\dal u_1=-(\pa_t u_3)(\pa_t u_2)+\text{(null forms)},\\
\Box u_2=(\pa_t u_3)(\pa_t u_1)+\text{(null forms)},\\
\dal u_3=\text{(null forms)}.
\end{cases}
\end{equation}
For this example, we can see that $C^{-1}\eps \le \|u(t)\|_E\le C\eps$ for
some positive constant $C$ unless $(f,g)\equiv(0,0)$; nonetheless, for appropriately chosen $f$ and $g$,
we can show that the global solution $u$ is {\it not} asymptotically free. 

Our aim in this paper is to obtain another kind of algebraic condition
which implies the small data global existence (and the weak null condition).
We will also show that, under this condition, we have the asymptotic behavior
that is quite different from the known cases.

\section{The main results}
In what follows, we assume the following condition on the nonlinearity:
\begin{enumerate}
\item[(H)] 
There is an $N\times N$-matrix valued continuous function 
${\mathcal A}={\mathcal A}(\omega)$ on $\Sph^2$ such that 
${\mathcal A}(\omega)$ is a positive-definite symmetric matrix for each 
$\omega\in \Sph^2$, and that
$$
 Y^{\rm T} {\mathcal A}(\omega) F^{\rm red}(\omega, Y)=0,
 \quad \omega\in \Sph^2,\ Y\in \R^N.
$$
\end{enumerate}

Concerning the global existence, our main result is the following:
\begin{thm}[Global existence]\label{thm_sdge}
Suppose that the condition {\rm (H)} is satisfied. 
Then, for any $f, g \in C_0^{\infty}(\R^3;\R^N)$, 
there exists $\eps_0>0$ such that 
\eqref{eq}--\eqref{data} admits a unique global $C^{\infty}$-solution
$u$ for $(t,x)\in [0,\infty) \times \R^3$ 
if $\eps \in (0,\eps_0]$. 
\end{thm}
Since the local existence of the solution is well known, what we have to do 
for the proof of Theorem~\ref{thm_sdge} is to get a suitable {\em a priori} 
estimate for the solution to \eqref{eq}--\eqref{data}. 
This will be carried out in Section~\ref{PT1} after some preliminaries 
in Sections~\ref{CVF} and \ref{PEQ}. 

Under the condition (H), 
there is a positive constant $M_0$ such that
\begin{equation}
\label{BBA}
M_0^{-1}|Y|^2 \le Y^{\rm T}{\mathcal A}(\omega)Y \le M_0|Y|^2,
\quad \omega\in \Sph^2.
\end{equation}
Indeed, if we denote the eigenvalues of
${\mathcal A}(\omega)$ by $\lambda_1(\omega), \ldots, \lambda_N(\omega)$ with
each eigenvalue being counted up to its algebraic multiplicity,
then we have
$$
 \min_{1\le j\le N} \lambda_j(\omega) |Y|^2
 \le 
 Y^{\rm T}{\mathcal A}(\omega) Y
 \le 
 \max_{1\le j\le N} \lambda_j(\omega) |Y|^2,
$$
which leads to \eqref{BBA} because we may assume $\lambda_j(\omega)$ is 
positive and continuous in $\omega\in \Sph^2$. 
Since \eqref{ReducedSystem} and (H) implies
$$
{\pa_t}\left(V^{\rm T}{\mathcal A}(\omega)V\right)=-\frac{1}{t}V^{\rm T}{\mathcal A}(\omega)F^{\rm red}(\omega, V)=0,
$$
we have an {\it a priori} bound for $|V|$ in view of \eqref{BBA}. Hence
the condition (H) implies the weak null condition.
If the null condition is satisfied, then the condition (H) is trivially 
satisfied with ${\mathcal A}(\omega)=I_N$, where $I_N$ is the identity 
$N\times N$-matrix. 
To sum up, the condition (H) is stronger than the weak null condition,
and weaker than the null condition. 

There is no inclusion between the
condition (H) and the (extended) Alinhac condition,
though both of them are satisfied for \eqref{ThirdExampleA};
the examples \eqref{FirstExampleA} and \eqref{SecondExampleA} 
satisfy the Alinhac condition but not the condition (H); 
the next examples satisfy the condition (H) but not the Alinhac 
condition.

\begin{example}\label{Typical}
Let $N=2$ and
\begin{equation}
\label{TypicalExample}
\begin{cases}
\displaystyle 
F_1(\pa u)= -c_0\sum_{a,b=0}^3 c_{ab} (\pa_a u_1)(\pa_b u_2)+N_1(\pa u),\\
\displaystyle 
F_2(\pa u)= \sum_{a,b=0}^3 c_{ab} (\pa_a u_1)(\pa_b u_1)+N_2(\pa u),
\end{cases}
\end{equation}
where $c_0$ is a positive constant and $c_{ab}$ are real constants, 
while $N_1$ and $N_2$ are written as linear combinations of the null forms. 
If we put 
\begin{equation}
 \label{Index}
 c(\omega)
 :=
 \sum_{a,b=0}^3 c_{ab}\omega_a\omega_b,
 \quad \omega=(\omega_1,\omega_2,\omega_3) \in \Sph^2
 \text{ with $\omega_0=-1$},
\end{equation}
then we get $F_1^{\rm red}(\omega, Y)=-c_0 c(\omega)Y_1Y_2$ and 
$F_2^{\rm red}(\omega, Y)=c(\omega)Y_1^2$. Hence we have
$$
 Y^{\rm T} {\mathcal A} F^{\rm red}(\omega, Y)
 =
 Y_1F_1^{\rm red}(\omega, Y)+c_0 Y_2F_2^{\rm red}(\omega, Y)
 =
 0
$$
with ${\mathcal A}=\diag(1, c_0)$,
and we see that the condition $(H)$ is satisfied.
The null condition is not satisfied unless $c(\omega)\equiv 0$.
Due to the result in \cite{ka}, we can also see that the Alinhac condition
is not satisfied unless $c(\omega)\equiv 0$, because the asymptotic behavior
as seen in Theorem~\ref{thm_asymp} below never happens under the Alinhac condition.
\end{example}

\begin{example}
Let $N=2$ and
\begin{equation}
\label{TypicalExampleR}
\begin{cases}
\displaystyle 
 F_1(\pa u)
 =
 -(\pa_t u_1)^2-4(\pa_t u_1)(\pa_t u_2)-3(\pa_t u_2)^2+(\pa_1u_1+\pa_1u_2)^2,
 \\
\displaystyle 
 F_2(\pa u)
 = 3(\pa_t u_1)^2+4(\pa_t u_1)(\pa_t u_2)+(\pa_tu_2)^2-(\pa_1u_1+\pa_1u_2)^2.
\end{cases}
\end{equation}
Then we can check that the condition (H) is satisfied with
$$
{\mathcal A}(\omega)
 =
 \frac{1}{2}
   \left(\begin{matrix}
        3-\omega_1^2 & 1-\omega_1^2\\
        1-\omega_1^2 & 3-\omega_1^2
   \end{matrix} \right).
$$
Indeed,
$$
  F^{\rm red}(\omega, Y)=(Y_1+Y_2)\left(
  \begin{matrix}
  (\omega_1^2-1)Y_1 + (\omega_1^2-3)Y_2\\
  (3-\omega_1^2)Y_1 + (1-\omega_1^2)Y_2
  \end{matrix}
  \right)
$$
is perpendicular to ${\mathcal A}(\omega)Y$. 
Note that the eigenvalues of ${\mathcal A}(\omega)$ are 
$1$ and $(2-\omega_1^2)$, both of which are positive for all 
$\omega\in \Sph^2$.
\end{example}

\begin{rmk}
Let us consider \eqref{eq}--\eqref{data} with $N=2$ and
\begin{equation}
\label{Simple}
\begin{cases}
F_1=-c_1(\pa_t u_1)(\pa_t u_2),\\
F_2=c_2(\pa_t u_1)^2.
\end{cases}
\end{equation}
We have small data global existence if $c_1c_2\ge 0$
because the condition (H) is satisfied for the case $c_1c_2>0$, while the case
$c_1c_2=0$ is trivial.
On the other hand,  if $c_1c_2<0$, then there are $f,g\in C^\infty_0(\R^3)$ 
such that the solution $u$ blows up in finite time no matter how small 
$\eps$ is. Indeed, if we choose $\phi, \psi\in C^\infty_0(\R^3)$ with 
$(\phi, \psi)\not \equiv (0,0)$, then the solution $u=(u_1, u_2)^{\rm T}$ with
$$
 (f_1,g_1)
 =
 \frac{1}{\sqrt{-c_1c_2}}(\phi, \psi),\ (f_2, g_2)=-\frac{1}{c_1}(\phi, \psi)
$$
can be written as $u_1=w/\sqrt{-c_1c_2}$ and $u_2=-w/c_1$,
where $w$ is the solution to $\Box w=(\pa_t w)^2$ with 
$(w,\pa_t w)(0)=(\eps \phi, \eps \psi)$, which blows up in finite time by the 
result of John~\cite{john}.
\end{rmk}

Now we give our main result for the asymptotic behavior of global solutions
under the condition (H). In order to keep the description not too 
complicated, we consider only the case of \eqref{TypicalExample} here. 
The case of general two-component systems satisfying (H) will be outlined 
in Section~\ref{ConclRem}. 
For simplicity of exposition, we put 
$\calH(\R^3):=\dot{H}^1(\R^3)\times L^2(\R^3)$ and
$$
\|(\phi,\psi)\|_{\calH}^2
:=\frac{1}{2}\left(\|\phi\|_{\dot{H}^1}^2+\|\psi\|_{L^2}^2\right),
\quad (\phi,\psi)\in \calH.
$$
Note that 
$\left\|\bigl(\varphi(t),\pa_t \varphi(t)\bigr)\right\|_{\calH}= \|\varphi(t)\|_E$.

\begin{thm}[Asymptotic behavior]\label{thm_asymp}
Let $N=2$, and assume that $F$ is of the form
\eqref{TypicalExample} with $c_0>0$. Suppose that 
$c(\omega)\not \equiv 0$ on $\Sph^2$,
where $c(\omega)$ is defined by \eqref{Index}.
Given $f=(f_1,f_2)^{\rm T}, g=(g_1, g_2)^{\rm T} \in C^\infty_0(\R^3;\R^2)$,
let $\eps$ be sufficiently small, and $u=(u_1,u_2)^{\rm T}$ be the 
global solution for \eqref{eq}--\eqref{data} whose existence is guaranteed by 
Theorem~$\ref{thm_sdge}$. 
Then we have 
\begin{align} 
 \lim_{t \to \infty} \|u_1(t)\|_E=0,
\label{asymp_u1}
\end{align}
and there exists $(f_2^{+},g_2^{+}) \in \calH(\R^3)$ such that 
\begin{align} 
 \lim_{t \to \infty} \|u_2(t)- u_2^{+}(t)\|_E=0,
\label{asymp_u2}
\end{align}
where $u_2^+=u_2^{+}(t,x)$ 
solves $\Box u_2^{+}=0$ with 
$(u_2^{+},\pa_t u_2^{+})(0)=(f_2^{+},g_2^{+})$.
Moreover we have
\begin{equation}
\label{AsympProfile}
\left\|(f_2^+, g_2^+)\right\|_{\calH}
=
\eps \left(
  c_0^{-1}\left\|(f_1, g_1)\right\|_{\calH}^2 
  + 
  \left\|(f_2,g_2)\right\|_{\calH}^2
\right)^{1/2}
+O(\eps^2)
\end{equation}
as $\eps\to +0$.
\end{thm}

The proof of Theorem~\ref{thm_asymp} will be given in Sections~\ref{ASPE} and 
\ref{PT2}. 
Note that the null condition is violated in the assumption of 
Theorem~\ref{thm_asymp} because we have assumed $c(\omega)\not \equiv 0$.

From \eqref{AsympProfile}, we see that $(f_2^{+}, g_2^{+}) \ne (0,0)$
for small $\eps$, unless the Cauchy data for the original problem vanish 
identically. 
Therefore Theorem \ref{thm_asymp} tells us that only $u_1$ is dissipated and 
$u_2$ behaves like a non-trivial free solution in the large time. 
As far as the authors know, there are no previous results on 
such decoupling in the context of nonlinear wave equations. 

Under the Alinhac condition, the global solution (at least for some data) 
behaves differently from the free solution in the large time unless the null 
condition is satisfied. 
In contrast, we may say that the global solution $u$ under the assumption of 
Theorem~\ref{thm_asymp} is asymptotically free by understanding 
\eqref{asymp_u1} as $\|u_1(t)-u_1^+(t)\|_E=0$ 
with $u_1^+\equiv 0$, which trivially satisfies $\Box u_1^+=0$; 
however, this case should be strictly distinguished from 
the situation under the null condition for the following reason. 
As we have stated in the introduction, if $N=2$ and the null condition is 
satisfied (that is $c(\omega)\equiv 0$ for \eqref{TypicalExample}), then
the solution $u=(u_1, u_2)^{\rm T}$ tends to $u^+=(u_1^+,u_2^+)^{\rm T}$ in 
the energy norm, where $u^+$ is the solution to $\Box u^+=0$ with some data 
$(u_j^+, \pa_t u_j^+)(0)=(f_j^+, g_j^+)$. Moreover we can easily obtain 
\begin{equation}
\label{freeProfile}
\left\|(f_j^+,g_j^+)-\eps(f_j,g_j)\right\|_{\calH}=O(\eps^2),\quad j=1,2
\end{equation}
as $\eps\to +0$, which shows that the effect of the nonlinearity is rather 
weak. \eqref{asymp_u1} and \eqref{AsympProfile} make a sharp contrast to 
\eqref{freeProfile}, and they are the consequence of the strong effect of the 
nonlinearity.

\begin{rmk}
Since the condition (H) is invariant under the change of variables 
$(t,x)\mapsto (-t,-x)$, we can also treat the backward Cauchy problem. Hence 
the existence of global $C^\infty$-solution for $(t,x)\in \R\times \R^3$ 
under the condition (H)
follows from Theorem~\ref{thm_sdge},
provided that $\eps$ is small enough. 
Since the form \eqref{TypicalExample} is
also invariant, we can apply Theorem~\ref{thm_asymp} to obtain
$$
\lim_{t\to -\infty} \|u_1(t)\|_E=0
$$
and
$$
\lim_{t\to -\infty} \|u_2(t)-u_2^-(t)\|_E=0
$$
for the solution $u_2^-$ to $\dal u_2^-=0$ with some data 
$(u_2^-,\pa_t u_2^-)(0)=(f_2^-, g_2^-)\in \calH(\R^3)$, provided that
the assumption of Theorem~\ref{thm_asymp} is fulfilled. 
Hence $u_1$ is dissipated not only forward but also backward in time.
\end{rmk}

\begin{rmk}
Here we mention some related topics:
\begin{itemize}
\item
Nonlinear Klein-Gordon systems in two space dimensions with nonlinearity of 
type~\eqref{TypicalExample} 
was considered in Kawahara-Sunagawa~\cite{ks} as an example violating the null 
condition for the Klein-Gordon systems (see \cite{dfx}, \cite{kos}, 
\cite{ks}, 
\cite{suna}, and references cited therein for the null condition for 
Klein-Gordon systems). 
\item
A system of nonlinear Schr\"odinger equations 
related to \eqref{Simple} was considered in
Hayashi-Li-Naumkin~\cite{hln} and \cite{hln2}, where one needs some 
restriction on the final state (see also \cite{li}, \cite{kls}). 
This might correspond to the situation in Theorem~\ref{thm_asymp} 
which suggests that the final state has the special form.
\end{itemize}
\end{rmk}

\section{Commuting vector fields} 
\label{CVF}

In this section, we recall basic properties of the vector fields 
associated with the wave equation. 
In what follows, we denote several positive constants by $C$ which may 
vary from one line to another. 
For  $y \in \R^d$ with a positive integer $d$, 
the notation $\jb{y}=(1+|y|^2)^{1/2}$ will be often used. 
Also we will use the following convention on implicit constants: 
The expression $f=\sum_{\lambda \in \Lambda}' g_{\lambda}$ means that there 
exists a family $\{A_{\lambda}\}_{\lambda \in \Lambda}$ of constants such 
that $f=\sum_{\lambda \in \Lambda} A_{\lambda} g_{\lambda}$.

Let us introduce 
\begin{align*} 
 &S=t\pa_t +\sum_{j=1}^{3}x_j \pa_j,
\\
 &L_j=t\pa_j+x_j\pa_t, \qquad\quad\ j\in \{1,2,3\},
\\
 &\Omega_{jk}
 =x_j\pa_k-x_k\pa_j,  \qquad j,k \in \{1,2,3\},
\\
 &\pa=(\pa_a)_{a=0,1,2,3}=(\pa_t,\pa_{x_1},\pa_{x_2},\pa_{x_3}),
\end{align*}
and we set 
$$
 \Gamma=(\Gamma_0,\Gamma_1, \ldots, \Gamma_{10})
 =
 (S,L_1,L_2,L_3,\Omega_{23},\Omega_{31},\Omega_{12},\pa_0,\pa_1,\pa_2,\pa_3).
$$
For a multi-index $\alpha=(\alpha_0,\alpha_1, \ldots, \alpha_{10})$, 
we write 
$\Gamma^{\alpha}
=\Gamma_0^{\alpha_0}\Gamma_1^{\alpha_1}\cdots\Gamma_{10}^{\alpha_{10}}$ 
and $|\alpha|=\alpha_0+\alpha_1+\cdots+\alpha_{10}$. 
We define 
$$
 |\phi(t,x)|_{k}=
 \left(\sum_{|\alpha|\le k} |\Gamma^{\alpha}\phi(t,x)|^2\right)^{1/2},
\quad 
 \|\phi(t,\cdot)\|_{k}
 =
 \left(
  \sum_{|\alpha|\le k} \|\Gamma^{\alpha}\phi(t,\cdot)\|_{L^2}^{2} 
 \right)^{1/2}
$$ 
for a non-negative integer $k$ and a smooth function $\phi=\phi(t,x)$. 
As is well known, these vector fields satisfy 
$[\Box,S]=2\Box$ and $[\Box,L_j]=[\Box,\Omega_{jk}]=[\Box,\pa_a]=0$,
where $[A,B]=AB-BA$ for linear operators $A$ and $B$. 
From them it follows that 
\begin{equation}
 \Box \Gamma^{\alpha} \phi =\widetilde{\Gamma}^{\alpha} \Box \phi,
\label{Comm01}
\end{equation}
where 
$\widetilde{\Gamma}^{\alpha}
=(\Gamma_0+2)^{\alpha_0}\Gamma_1^{\alpha_1}\cdots\Gamma_{10}^{\alpha_{10}}$. 
We also note that 
$$
 [\Gamma_j,\Gamma_k]= 
\Sum_{l=0}^{10}\Gamma_l, 
\quad
 [\Gamma_j,\pa_a]= 
\Sum_{b=0}^{3}\pa_b.
$$
Hence we can check that the estimates 
\begin{align}
 & |\Gamma^{\alpha}\Gamma^{\beta}\phi| \le C |\phi|_{|\alpha|+|\beta|}, 
 \nonumber\\
 & C^{-1} |\pa \phi|_{s} \le 
   \sum_{|\alpha|\le s} |\pa \Gamma^{\alpha} \phi|
   \le C |\pa \phi|_{s} \label{Comm03}
\end{align}
are valid for any multi-indices $\alpha$, $\beta$ and any non-negative integer 
$s$.

Next we set $r=|x|$, $\omega_j=x_j/r$, $\pa_r=\sum_{j=1}^{3}\omega_j \pa_j$,
and $\pa_{\pm}=\pa_t \pm \pa_r$. 
We write $\omega=(\omega_j)_{j=1,2,3}$. 
For simplicity of exposition, we also introduce
$$
D_\pm=\pm\frac{1}{2}\pa_\pm=\frac{1}{2}(\pa_r\pm \pa_t).
$$
We summarize several useful inequalities related to $\Gamma$. 

\begin{lem} \label{lem3_01}
For a smooth function $\phi$ of $(t,x) \in [0,\infty)\times \R^3$, we have 
\begin{align}
 &\left|D_{+}(r\phi(t,x)) \right|
 \le 
 C |\phi(t,x)|_{1},
 \label{est_d_plus}\\
 &\left|r\pa_t \phi(t,x) 
{}+D_-(r\phi(t,x)) \right|
 \le 
 C |\phi(t,x)|_{1},
 \label{est_pa_t}
\end{align}
and 
\begin{align}
 &\left|r \pa_j \phi(t,x) 
-\omega_jD_-(r\phi(t,x)) \right|
 \le 
 C |\phi(t,x)|_{1}
 \label{est_pa_j}
\end{align}
for $j=1,2,3$.
\end{lem}

\begin{proof}
\eqref{est_d_plus} and \eqref{est_pa_t} are direct consequences of the 
following relations:  
\begin{align*}
 & D_+(r\phi)=\frac{r}{2(r+t)} (S\phi+L_r\phi)+\frac{\phi}{2},\\
&r\pa_t \phi =
{}-D_-(r\phi)+D_+(r\phi),
\end{align*}
where $L_r=r\pa_t+t\pa_r= \sum_{j=1}^{3}\omega_j L_j$. 
\eqref{est_pa_j} follows just from
\begin{align}
 r(\pa_j -\omega_j\pa_r)\phi = \sum_{k=1}^{3} \omega_k \Omega_{kj}\phi
\label{est_perp}
\end{align}
and
\begin{align*}
 r\pa_r \phi = 
D_-(r\phi)+D_+(r\phi)-\phi,
\end{align*}
if we use \eqref{est_d_plus} to estimate $D_+\phi$.
\end{proof}

\begin{lem}\label{lem3_02}
 For a smooth function $\phi$ of $(t,x) \in [0,\infty)\times \R^3$ 
and a non-negative integer $s$, we have 
\begin{align*}
 |\pa \phi(t,x)|_s \le C \jb{t-|x|} ^{-1} |\phi(t,x)|_{s+1}.
\end{align*}
\end{lem}
This lemma is due to Lindblad~\cite{lind}, which comes from the identities 
$$
(t-r)\pa_t\phi=\frac{1}{t+r}\left(tS-rL_r\right)\phi,
$$
$$
(t-r)\pa_r\phi=\frac{1}{t+r}\left(tL_r-rS\right)\phi,
$$
and $t\Omega_{kj}\phi=x_kL_j\phi-x_jL_k\phi$, as well as \eqref{est_perp}
(see \cite{lind} for the detail of the proof).
\medskip

We close this section with the following decay estimate 
for solutions to inhomogeneous wave equations. 

\begin{lem}[H\"ormander's $L^{1}$--$L^{\infty}$ estimate]\label{lem3_03}
Let $\phi$ be a smooth solution to 
$$
 \Box \phi =G, \quad (t,x) \in (0,T)\times \R^3
$$ 
with $\phi(0,x)=\pa_t \phi(0,x)=0$. It holds that 
$$
 \jb{t+|x|} |\phi(t,x)|
 \le 
 C
 \sum_{|\alpha|\le 2} \int_{0}^{t} 
 \|\Gamma^{\alpha}G(\tau,\cdot)\|_{L^1(\R^3)} \frac{d\tau}{\jb{\tau}}
$$
for $0\le t<T$. Here the constant $C$ is independent of $T$.
\end{lem}
%
See \cite{hor} for the proof (see also Lemma~6.6.8 of \cite{hor2}, 
or Section~2.1 of \cite{sogge}). 
\begin{rmk}\label{HomDecay} Various kinds of decay estimates for homogeneous wave equations are also
available. Here we only mention the following one that is a simple 
corollary to Lemma~\ref{lem3_03} via the cut-off argument (see \cite{Lin08} for the proof): 
For $R>0$, there is a positive constant $C_R$ such that we have
$$
\jb{t+|x|}|\phi(t,x)|\le C_R \|\pa\phi(0)\|_2
$$
for a smooth solution $\Box \phi(t,x)=0$ for $(t,x)\in (0,\infty)\times \R^3$,
provided that $\phi(0,x)=(\pa_t\phi)(0,x)=0$ for $|x|\ge R$.
\end{rmk}

\section{The profile equation} 
\label{PEQ}

Let $0<T\le \infty$, and let  $u$ be the solution to \eqref{eq}--\eqref{data}
on $[0, T)\times \R^3$.
We suppose that 
\begin{align}
 \supp f\cup \supp g \subset B_R 
\label{supp_0}
\end{align}
for some $R>0$, where $B_M=\{x\in \R^3;|x|\le M\}$ for $M>0$. 
Then, from the finite propagation property, we have
\begin{align}
 \supp u(t,\cdot)\subset B_{t+R},\quad 0\le t<T.
\label{supp_t}
\end{align}

Now we put $r=|x|$, $\omega=(\omega_1, \omega_2,\omega_3)=x/|x|$ 
and set 
$$
 \Delta_{\Sph^2}=\sum_{1\le j<k\le 3} \Omega_{jk}^2,
$$
so that 
\begin{align} 
r\Box \phi =\pa_+ \pa_-(r\phi)-\frac{1}{r}\Delta_{\Sph^2} \phi.
 \label{dal_polar}
\end{align}
We define $U=(U_1,\ldots, U_N)^{\rm T}$ by
\begin{align} 
U(t, x):=D_-\bigl(r u(t, x) \bigr),
\quad (t, x)\in [0,T)\times (\R^3\setminus\{0\})
 \label{U}
\end{align}
for the solution $u$ of \eqref{eq}. In view of 
\eqref{est_pa_t} and \eqref{est_pa_j}, the asymptotic profiles as 
$t \to \infty$ of $\pa_t u$ and $\nabla_x u$ should be given by $ -U/r$ and 
$\omega U/r$, respectively, because we can expect $|u(t,x)|_1\to 0$ as 
$t\to\infty$. Also it follows from \eqref{dal_polar} that 
\begin{align}
 \pa_+U(t, x)= -\frac{1}{2t} F^{\rm red}\bigl(\omega, U(t,x)\bigr)+H(t, x),
 \label{ode_0}
\end{align}
where $F^{\rm red}=F^{\rm red}(\omega, Y)$ is defined by \eqref{ReducedNonlinearity},
and $H=H(t,x)$ is given by
\begin{align}
\label{DefRemainder}
H=& 
-\frac{1}{2}\left(rF(\pa u)-\frac{1}{t}F^{\rm red}(\omega, U)\right)
-\frac{1}{2r} \Delta_{\Sph^2} u.
\end{align}
As we will see in Lemma~\ref{lem4_01} below, $H$ can be regarded as a 
remainder. For these reasons, we call \eqref{ode_0} {\em the profile equation} 
associated with \eqref{eq}, which plays an important role 
in our analysis. Observe that the reduced system \eqref{ReducedSystem} is
obtained by neglecting $H$ and changing variables in \eqref{ode_0} (see \eqref{ode_V} below).

We also need an analogous equation for $\Gamma^\alpha u$ with a multi-index 
$\alpha$. For this purpose, we put
\begin{equation}
\label{DefUAlpha}
U^{(\alpha)}(t,x)=\bigl(U_1^{(\alpha)}(t,x),\ldots, U_N^{(\alpha)}(t,x)\bigr)^{\rm T}:= D_-\bigl(r\Gamma^\alpha u(t,x)\bigr).
\end{equation}
Since 
$\dal(\Gamma^\alpha u)=\widetilde{\Gamma}^\alpha\left(F(\pa u)\right)$,
we deduce from \eqref{dal_polar} that 
\begin{align}
\pa_+U^{(\alpha)}= -\frac{1}{2t} G_{\alpha}(\omega, U, U^{(\alpha)})+H_\alpha
 \label{ode_alpha}
\end{align}
for $|\alpha|\ge 1$,
where $G_\alpha=(G_{\alpha,j})_{1\le j\le N}^{\rm T}$ is given by
\begin{align}
 G_{\alpha,j}\left(\omega, U, U^{(\alpha)}\right)
 =& \sum_{k=1}^N\frac{\pa F_j^{\rm red}}{\pa Y_k}(\omega, U) U^{(\alpha)}_k 
 \nonumber\\
 =& \sum_{k,l=1}^N\sum_{a,b=0}^3 c^{kl, ab}_j\omega_a\omega_b\left(U_k U_l^{(\alpha)}+U_k^{(\alpha)}U_l\right)
 \label{DefGAlpha}
\end{align}
with the constants $c^{kl, ab}_j$ appeared in \eqref{Nonlinearity},
and $H_\alpha$ is given by 
\begin{equation}
H_\alpha(t,x)
= 
-\frac{1}{2}\left(
 r\widetilde{\Gamma}^\alpha F(\pa u)- \frac{1}{t}G_{\alpha}\left(\omega, U, U^{(\alpha)}\right)
\right)
{}-\frac{1}{2r}\Delta_{\Sph^2} \Gamma^\alpha u.
\label{DefHAlpha}
\end{equation}

In the rest part of this section, we focus on preliminary estimates for 
$H$ and $H_{\alpha}$ in terms of the solution $u$ near the light cone. 
To be more specific, we put
\begin{align*}
 \Lambda_{T,R}:=\{(t,x)\in [0, T)\times \R^3;\, 1\le t/2\le |x|\le t+R\}.
\end{align*}
Note that we have
$$
(1+t+|x|)^{-1}\le |x|^{-1}\le 2t^{-1}\le 3(1+t)^{-1}\le 3(R+2)(1+t+|x|)^{-1}
$$
for $(t,x)\in \Lambda_{T,R}$. In other words, the weights $\jb{t+|x|}^{-1}$, 
$(1+t)^{-1}$, $|x|^{-1}$, and
$t^{-1}$ are equivalent to each other in $\Lambda_{T,R}$.
For a non-negative integer $s$, we also introduce an auxiliary notation 
$|\cdot |_{\sharp,s}$ by
\begin{equation}
|\phi(t,x)|_{\sharp,s}:=|\pa \phi(t,x)|_s+\jb{t+|x|}^{-1}|\phi(t,x)|_{s+1}.
\label{norm_sharp}
\end{equation}

\begin{lem}\label{lem4_01}
We have 
\begin{align*}
|H(t,x)|
 \le & C \left(|u(t,x)|_{\sharp,0}|u(t,x)|_1+t^{-1}|u(t,x)|_2 \right),
\end{align*}
for $(t,x)\in \Lambda_{T,R}$. 
Here the constant $C$ is independent of $T$. 
Also, in the case of $s \ge 1$, we have 
\begin{align*}
 \sum_{|\alpha|=s}|H_\alpha(t,x)|\le & C_s (|u|_{\sharp,s}|u|_{s+1}
 +t|\pa u|_{s-1}^2+t^{-1}|u|_{s+2})
\end{align*}
for $(t,x)\in \Lambda_{T,R}$, where $C_s$ is a positive constant 
which does not depend on $T$.
\end{lem}

\begin{proof}
Let $(t,x)=(t,r\omega)\in \Lambda_{T,R}$ and $|\alpha|=s \ge 0$.
First we note that 
\begin{align}
|U^{(\alpha)}(t,x)|
 &\le 
 r|D_-\Gamma^{\alpha}u|+\frac{1}{2}|\Gamma^{\alpha}u|
 \nonumber\\
 &\le 
 C r (|\pa u|_{s} + r^{-1}|u|_{s})
 \le 
 C t|u(t,x)|_{\sharp,s}
\label{est_U_alpha_01}
\end{align}
by the definition of $|\cdot|_{\sharp,s}$, and that 
\begin{align}
 \jb{t-r}|U^{(\alpha)}(t,x)|
 &\le 
 C t \left( 
  \jb{t-r}|\pa u(t,x)|_s+ \frac{\jb{t-r}}{\jb{t+r}}|u(t,x)|_{s} 
 \right) \nonumber\\
 &\le 
 C t |u(t,x)|_{s+1}
\label{est_U_alpha_02}
\end{align}
by Lemma~\ref{lem3_02}. 

Now we consider the estimate for $H$. We decompose it as follows: 
$$
 H
 =
 -\frac{1}{2r}\Bigl( r^2 F(\pa u) - F^{\rm red}(\omega, U) \Bigr) 
 {}-\frac{t-r}{2rt}F^{\rm red}(\omega, U)
 {}-\frac{1}{2r} \Delta_{\Sph^2} u.
$$
It is easy to see that the third term can be dominated by $Ct^{-1}|u|_2$. 
As for the second term, we have 
\begin{align*}
 \frac{|t-r|}{rt}|F^{\rm red}(\omega, U)|
 \le 
 C t^{-1} \jb{t-r} |U| \cdot t^{-1}|U|
 \le 
 C |u|_{1} |u|_{\sharp,0}, 
\end{align*}
because of \eqref{est_U_alpha_01} and \eqref{est_U_alpha_02} with $s=0$. 
To estimate the first term, 
noting that \eqref{est_pa_t} and \eqref{est_pa_j} imply
\begin{align*}
|(r\pa_a u_k)(r\pa_b u_l)-(\omega_a U_k)(\omega_b U_l)|
\le & |r\pa_a u_k-\omega_aU_k|\,|r\pa_b u_l|+|\omega_aU_k|\,|r\pa_bu_l-\omega_bU_l|\\
\le & C(r|\pa u|+|U|)|u|_1
\end{align*}
with $\omega_0=-1$,
and that \eqref{Nonlinearity} yields
\begin{align*}
 r^2 F_j(\pa u) - F_j^{\rm red}(\omega, U) 
 &=
  \sum_{k,l=1}^N\sum_{a,b=0}^3 c_j^{kl, ab} \left((r\pa_a u_k)(r\pa_b u_l)-(\omega_a U_k)(\omega_b U_l)\right),
\end{align*}
we obtain 
\begin{align*}
 \frac{1}{2r}| r^2 F(\pa u) - F^{\rm red}(\omega, U)|
 \le 
 C (|\pa u| + r^{-1}|U|)|u|_1
 \le 
 C|u|_{\sharp,0}|u|_1
\end{align*}
with the help of \eqref{est_U_alpha_01}.

Next we turn to the estimate for $H_{\alpha}$ with $|\alpha|=s \ge 1$. 
For this purpose, we set $\widetilde{F}_\alpha=(\widetilde{F}_{\alpha, j})_{1\le j\le N}^{\rm T}$ with 
$$
 \widetilde{F}_{\alpha, j}
 =\sum_{k,l=1}^N \sum_{a,b=0}^3 c_j^{kl, ab}\left((\pa_a u_k)(\Gamma^\alpha \pa_b u_l)+
 (\Gamma^\alpha \pa_a u_k)(\pa_b u_l)\right)
$$
to split $H_{\alpha}$ into the following form: 
$$
 H_{\alpha}
 =
 -\frac{r}{2} \left(\widetilde{\Gamma}^{\alpha}F(\pa u) 
 -\widetilde{F}_{\alpha} \right) 
 -\frac{1}{2r}
 \left(r^2 \widetilde{F}_{\alpha} - G_{\alpha} \right) 
 -\frac{t-r}{2rt}G_{\alpha} 
 -\frac{1}{2r} \Delta_{\Sph^2} \Gamma^{\alpha}u.
$$
Since the first term consists of a linear combination of the terms 
in the form $r(\Gamma^{\beta}\pa_a u_k)(\Gamma^{\gamma}\pa_b u_l)$ 
with $|\beta|, |\gamma|\le s-1$, $k,l\in \{1,\ldots, N\}$, and $a, b \in \{0,1,2,3\}$, 
it can be estimated by $Ct|\pa u|_{s-1}^2$. Other terms can be treated 
in the same way as in the previous case.
\end{proof}

\section{Proof of Theorem \ref{thm_sdge}}
\label{PT1} 
 
Let $u(t,x)$ be a smooth solution to 
\eqref{eq}--\eqref{data} on $[0, T_0)\times\R^3$ with some 
$T_0\in (0,\infty]$. For $0<T\le T_0$, we put 
\begin{align*}
 e[u](T)=& \sup_{(t,x)\in[0,T)\times \R^3}
 \Bigl(
 \jb{t+|x|} \jb{t-|x|}^{1-\mu}  |\pa u(t,x)| 
 \\
 &\qquad\qquad\qquad\qquad {}+ 
 \jb{t+|x|}^{1-\nu}  \jb{t-|x|}^{1-\mu}
 |\pa u(t,x)|_{k} 
 \Bigr)
\end{align*}
with some $\mu$, $\nu>0$ and a positive integer $k$.
We also put
$$
e[u](0)=\lim_{T\to +0} e[u](T).
$$
Observe that there is a positive constant $\eps_1$ such that 
$0<\eps\le \eps_1$ implies $e[u](0)\le \sqrt{\eps}/2$, 
because we have $e[u](0)=O(\eps)$.

The main step toward global existence is to show the following. 

\begin{prp}[{\em A priori} estimate]\label{lem_apriori}
Let $k\ge 3$, $0<\mu < 1/2$, and $0<4(k+1)\nu \le \mu$. 
There exist positive constants $\eps_2$ and $m$, which depend only on 
$k$, $\mu$ and $\nu$,  such that 
\begin{align} 
 e[u](T)\le \sqrt{\eps}
\label{est_before}
\end{align} 
implies 
\begin{align} 
 e[u](T) \le m\eps, 
\label{est_after}
\end{align} 
provided that $0<\eps \le \eps_2$ and $0<T\le T_0$.
\end{prp}

Once the above proposition is obtained, we can show the small data global existence 
for \eqref{eq}--\eqref{data} by the so-called continuity 
argument: 
Let $T^*$ be the lifespan of the classical solution for 
(\ref{eq})--(\ref{data}) and assume $T^*<\infty$. 
Then, it follows from the standard blow-up criterion (see e.g., \cite{sogge}) 
that 
\begin{align}
 \lim_{t\to T^*-0} \|\pa u(t, \cdot)\|_{L^\infty(\R^3)}=\infty.
 \label{blowup}
\end{align}
On the other hand, by setting
$$
T_*=\sup\left\{T\in [0,T^*)\,; e[u](T)\le \sqrt{\eps}  \right\},
$$
we can see that Proposition~\ref{lem_apriori} yields $T_*=T^*$,
provided that $\eps$ is small enough. Indeed, if $T_*<T^*$, then we have $e[u](T_*)\le\sqrt{\eps}$, and
Proposition~\ref{lem_apriori} implies that
$$
 e[u](T_*)\le m \eps \le \frac{\sqrt{\eps}}{2}
$$
for $0<\eps\le \min\{\eps_1, \eps_2, 1/4m^2\}$
(note that we have $T_*>0$ for $\eps\le \eps_1$). 
Then, by the continuity of  $[0,T^*) \ni T \mapsto e[u](T)$, 
we can take $\delta>0$ such that 
$$
 e[u](T_*+\delta) \le \sqrt{\eps}, 
$$
which contradicts the definition of $T_*$, and we conclude that $T_*=T^*$. 

In particular, we have 
$$
 e[u](T^*) \le \sqrt{\eps}. 
$$
This implies that \eqref{blowup} never occurs for small $\eps$. 
In other words, we must have $T^*=\infty$, 
that is, the solution $u$ exists globally for small data. 
We also note that
\begin{equation}
\label{AE00}
e[u](\infty) \le \sqrt{\eps}
\end{equation}
holds for this global solution $u$, and
Proposition~\ref{lem_apriori} again yields
\begin{equation}
 \label{AE}
 e[u](\infty)\le m\eps.
\end{equation}

Now we turn to the proof of Proposition~\ref{lem_apriori}.  
It will be divided into several steps.
\begin{proof}[Proof of Proposition~\ref{lem_apriori}]
In what follows, we always suppose $0\le t<T$. 

\noindent{\bf Step 1: 
Rough bounds for $|u(t,x)|_{k+2}$ and 
$|\pa u(t,x)|_{k+1}$.}

First of all, we will establish the following energy estimates: 
\begin{align}
 \|\pa u(t)\|_{l} \le C \eps (1+t)^{C_* \sqrt{\eps} + l\nu}
 \label{est_energy}
\end{align}
for $l\in \{0,1, \ldots, 2k+1\}$, where $C_*$ is a positive constant 
to be fixed later. 

In preparation for the proof of \eqref{est_energy}, we make some observations: 
Let $0\le l\le 2k+1$.
In what follows we neglect terms including $|\pa u|_{l-1}$ or $\|\pa u\|_{l-1}$
when $l=0$.
From \eqref{Comm01}, \eqref{Comm03}, and the standard energy inequality, we get
\begin{equation}
 \|\pa u(t)\|_l
 \le 
 C_{1,l}\left(\|\pa u(0)\|_l
 +
 \int_0^t\left\|F\bigl(\pa u(\tau)\bigr)\right\|_l d\tau\right),
 \label{Ene01}
\end{equation}
where $C_{1,l}$ is a positive constant depending only on $l$.
From \eqref{est_before} we have
 $|\pa u(t,x)|
 \le 
 \sqrt{2\eps}(1+t)^{-1}$ and $|\pa u(t,x)|_k\le \sqrt{2\eps}(1+t)^{\nu-1}$, 
since $\jb{t+|x|}^{-1} \le \sqrt{2} (1+t)^{-1}$.
Hence we get
\begin{align*}
|F(\pa u)|_l
 \le & 
 C_{2,l}\left(|\pa u|\,|\pa u|_l+|\pa u|_{[l/2]}|\pa u|_{l-1}\right) \\
 \le & 
 C_{2,l}\sqrt{2\eps}\left((1+t)^{-1}|\pa u|_l+(1+t)^{\nu-1}|\pa u|_{l-1}\right)
\end{align*}
with a positive constant $C_{2,l}$ depending only on $l$, which leads to
\begin{equation}
 \left\|F\bigl(\pa u(t)\bigr)\right\|_l
 \le 
 \sqrt{2}C_{2,l}\sqrt{\eps} \left((1+t)^{-1}\|\pa u(t)\|_l
 +
 (1+t)^{\nu-1}\|\pa u(t)\|_{l-1}\right).
\label{Ene02}
\end{equation}

Now we put $C_*=\max_{0\le l\le 2k+1} \sqrt{2}C_{1,l} C_{2, l}$, 
and we shall prove \eqref{est_energy} by induction on $l$. 
If $l=0$, it follows from \eqref{Ene01} and \eqref{Ene02} that
\begin{align*}
 \|\pa u(t)\|_{0} 
 &\le 
 C\eps + C_*\sqrt{\eps} 
 \int_0^t (1+\tau)^{-1}\|\pa u(\tau)\|_{0}d\tau,
\end{align*}
whence the Gronwall lemma implies 
$$
 \|\pa u(t)\|_0 
 \le 
  C\eps (1+t)^{C_*\sqrt{\eps}}.
$$
Next we assume that \eqref{est_energy} holds for some 
$l\in \{0,1, \ldots, 2k\}$. 
Then it follows from \eqref{Ene01} and \eqref{Ene02} that
\begin{align*}
 \|\pa u(t)\|_{l+1} 
 \le & 
 C\eps+C_*\sqrt{\eps}\int_{0}^{t} \left((1+\tau)^{-1}\|\pa u(\tau)\|_{l+1}
 {}+(1+\tau)^{-1+\nu} \|\pa u(\tau)\|_{l}\right)
 d\tau
\\
 \le &
 C\eps + C_*\sqrt{\eps} \int_{0}^{t} (1+\tau)^{-1}\|\pa u(\tau)\|_{l+1}d\tau\\
 & {}+C \eps^{3/2} \int_{0}^{t} (1+\tau)^{-1+C_*\sqrt{\eps}+(l+1)\nu}d\tau\\
 \le & 
 C\eps + C_*\sqrt{\eps} \int_{0}^{t} (1+\tau)^{-1}\|\pa u(\tau)\|_{l+1}d\tau
 + C \eps^{3/2}(1+t)^{C_*\sqrt{\eps}+(l+1)\nu},
\end{align*}
which yields
\begin{align*}
  \|\pa u(t)\|_{l+1} 
 \le 
 C\eps (1+t)^{C_* \sqrt{\eps}}
 + 
 C\eps^{3/2} (1+t)^{C_*\sqrt{\eps}+(l+1)\nu}
 \le 
 C\eps (1+t)^{C_* \sqrt{\eps}+(l+1)\nu}.
\end{align*}
This means that \eqref{est_energy} remains true when $l$ is replaced by $l+1$, 
and \eqref{est_energy} has been proved for all $l\in \{0,1,\ldots, 2k+1\}$. 

From now on, we assume  that $\eps\le \nu^2/C_*^2$. 
Then, since $k\ge 3$ and $2(k+1)\nu \le \mu/2$, 
it follows from \eqref{est_energy} with $l=2k+1$ that 
\begin{equation}
 \|\pa u(t)\|_{k+4} 
 \le 
 \|\pa u(t)\|_{2k+1} 
 \le 
 C \eps \jb{t}^{2(k+1)\nu}
 \le 
 C \eps \jb{t}^{\mu/2}
\label{Ene03}
\end{equation}
and
$$
 \bigl\| |F(\pa u(t,\cdot))|_{k+4} \bigr\|_{L^1(\R^3)} 
 \le  C\|\pa u(t)\|_{k+4}^2 
\le C \eps^2 \jb{t}^{\mu}.
$$
Hence Lemma~\ref{lem3_03} and Remark~\ref{HomDecay} yield 
\begin{align*}
 \jb{t+|x|} |u(t,x)|_{k+2} 
 \le & C_R\|\pa u(0)\|_{k+4}+C\int_0^t \frac{\bigl\| |F(\pa u(\tau))|_{k+4}\bigr\|_{L^1}}{\jb{\tau}}d\tau\\
\le & C \eps + C\eps^2  \int_0^t  \jb{\tau}^{\mu-1}d\tau
 \le 
 C\eps \jb{t+|x|}^{\mu}, 
\end{align*} 
that is,
\begin{align}
 |u(t,x)|_{k+2} \le C \eps \jb{t+ |x|}^{-1+\mu}
\label{rough_bound_u}
\end{align}
for $(t,x) \in [0,T)\times \R^3$. By Lemma~\ref{lem3_02}, we also have 
\begin{align}
 |\pa u(t,x)|_{k+1} \le C \eps \jb{t+|x|}^{-1+\mu} \jb{t-|x|}^{-1}
\label{rough_bound_pa_u}
\end{align}
for $(t,x) \in [0,T)\times \R^3$. 
\medskip

\noindent{\bf Step 2: 
Estimates for $|\pa u(t,x)|_{k}$ away from the light cone.}

Now we put 
$\Lambda_{T,R}^{\rm c}:=\bigl([0,T)\times \R^3\bigr)\setminus \Lambda_{T,R}$,
where $R$ is the constant appearing in \eqref{supp_0}.
In the case of $t/2<1$ or $|x|< t/2$, we see that 
$$
\jb{t-|x|}\le \jb{t+|x|}\le C\jb{t-|x|}.
$$
On the other hand, it follows from \eqref{supp_t} that $u(t,x)=0$ if $|x|>t+R$.
Hence \eqref{rough_bound_pa_u} implies
\begin{align}
 \sup_{(t,x)\in \Lambda_{T,R}^c} \jb{t+|x|} \jb{t-|x|}^{1-\mu}|\pa u(t,x)|_{k}
 \le C\eps.
\label{est_after_01}
\end{align}
\medskip

\noindent{\bf Step 3: 
Estimates for $|\pa u(t,x)|$ near the light cone.}

Let $(t,x) \in \Lambda_{T,R}$ throughout this step.
Remember that $t^{-1}$, $|x|^{-1}$, $\jb{t}^{-1}$, and $\jb{t+|x|}^{-1}$ are 
equivalent to each other in $\Lambda_{T,R}$. 
We define $U$, $U^{(\alpha)}$, $H$, $H_\alpha$ and $|\,\cdot\,|_{\sharp, s}$ 
as in Section~\ref{PEQ} (see \eqref{U}, \eqref{DefRemainder}, \eqref{DefUAlpha},
\eqref{DefHAlpha}, and \eqref{norm_sharp}). 
We see from  \eqref{rough_bound_u} and \eqref{rough_bound_pa_u} that 
\begin{align}
 |u(t,x)|_{\sharp,k}\le C\eps t^{\mu-1}\jb{t-|x|}^{-1}.
 \label{est_u_sharp01}
\end{align}
By \eqref{Comm03}, \eqref{est_pa_t}, \eqref{est_pa_j}, 
and \eqref{rough_bound_u},  we have
\begin{align}
t|\pa u(t,x)|_l
 \le & 
 C \sum_{|\alpha|\le l} \bigl|\,|x| \pa \Gamma^\alpha u(t,x) \bigr| 
 \nonumber\\
 \le & 
 C\sum_{|\alpha|\le l} |U^{(\alpha)}(t,x)| + C\eps t^{\mu-1}
 \label{est_pa_u_02}
\end{align}
for $l\le k$. 
Also, it follows from \eqref{rough_bound_u}, \eqref{est_u_sharp01}, and 
Lemma~\ref{lem4_01} that 
\begin{align}
|H(t,x)|
 \le 
 C\left(\eps^2 t^{2\mu-2}\jb{t-|x|}^{-1}
 + 
 \eps t^{\mu-2}\right)
 \le 
C\eps t^{2\mu-2}\jb{t-|x|}^{-\mu}.
 \label{est_H}
\end{align}
Next we put 
$$
\Sigma=\{(t,x) \in \Lambda_{T,R}\, ;\, t/2=1\ \ \mbox{or}\ \ t/2=|x| \}
$$
and we define $t_{0, \sigma}=\max\{2, -2\sigma\}$ for $\sigma\le R$.  
What is important here is that the line segment 
$\{(t, (t+\sigma)\omega);\, 0\le t< T\}$ meets $\Sigma$ at the point 
$\left(t_{0, \sigma}, (t_{0,\sigma}+\sigma)\omega\right)$ for each 
fixed $(\sigma,\omega) \in (-\infty, R]\times \Sph^2$.  
We also remark that 
\begin{align}
 C^{-1}\jb{\sigma} \le t_{0,\sigma} \le C \jb{\sigma}, \quad \sigma\le R.
 \label{t_zero}
\end{align}
When $(t,x)\in \Sigma$, 
we have $t^{\mu}\le C\jb{t-|x|}^{\mu}$. 
So it follows from \eqref{est_U_alpha_01} and \eqref{est_u_sharp01} that 
\begin{align}
 \sum_{|\alpha|\le k}|U^{(\alpha)}(t,x)|
 \le 
 C\eps t^{\mu} \jb{t-|x|}^{-1}
 \le 
 C\eps \jb{t-|x|}^{\mu-1},
\quad (t,x)\in \Sigma.
\label{est_U_alpha}
\end{align}

Now we define
\begin{align}
V(t; \sigma,\omega)=\bigl(V_1(t;\sigma,\omega),\ldots,V_N(t;\sigma,\omega)\bigr)^{\rm T}=U\bigl(t, (t+\sigma)\omega\bigr)
\label{profile_v}
\end{align}
for $0\le t< T$ and 
$(\sigma,\omega) \in \R\times \Sph^2$. In what follows, we fix $(\sigma, \omega)\in (-\infty, R]\times \Sph^2$ and write $V(t)$ for $V(t;\sigma, \omega)$. 
Then, since the profile equation \eqref{ode_0} is rewritten as 
\begin{align}
 \frac{\pa V}{\pa t}(t)=(\pa_+U)\bigl(t, (t+\sigma)\omega\bigr)
 = 
 -\frac{1}{2t}F^{\rm red}(\omega, V(t))+H(t, (t+\sigma)\omega) 
 \label{ode_V}
\end{align}
for $t_{0,\sigma} < t< T$, it follows from the condition (H) that 
\begin{align}
 \frac{\pa}{\pa t}\left(V(t)^{\rm T}{\mathcal A}(\omega)V(t)\right)
 &=2 V(t)^{\rm T} {\mathcal A}(\omega)
   \frac{\pa V}{\pa t}(t) \nonumber\\
   &=2V(t)^{\rm T}{\mathcal A}(\omega)\left(-\frac{1}{2t}F^{\rm red}\bigl(\omega, V(t)\bigr)+H\bigl(t, (t+\sigma)\omega\bigr)\right)\nonumber\\ 
 &=2 V(t)^{\rm T}{\mathcal A}(\omega) H(t,(t+\sigma)\omega)
    \nonumber\\
 &\le 
 C|V(t)|\,|H(t,(t+\sigma)\omega)|\nonumber\\
 &\le C\sqrt{V(t)^{\rm T}{\mathcal A}(\omega)V(t)}|H(t,(t+\sigma)\omega)|
\label{star02}
\end{align}
for $t_{0,\sigma} < t< T$,
where we have used \eqref{BBA} to obtain the last line.
We also note that \eqref{est_U_alpha} for $k=0$ can be interpreted as
\begin{align}
 |V(t_{0,\sigma})|=\left|U\bigl(t_{0,\sigma}, (t_{0,\sigma}+\sigma)\omega\bigr)\right|
 \le 
 C\eps \jb{\sigma}^{\mu-1}.
 \label{est_V_ini}
\end{align}
From \eqref{BBA}, \eqref{est_H}, \eqref{t_zero}, \eqref{star02}, and 
\eqref{est_V_ini} we get
\begin{align}
 |V(t)| & \le \sqrt{M_0} \sqrt{V(t)^{\rm T}{\mathcal A}(\omega)V(t)} 
 \nonumber\\
 &\le C\left(
 \sqrt{V(t_{0,\sigma})^{\rm T}{\mathcal A}(\omega)V(t_{0,\sigma})} 
 +
 \int_{t_{0,\sigma}}^{t} |H(\tau,(\tau + \sigma)\omega)|d\tau\right)
 \nonumber\\
 &\le 
 C\eps \jb{\sigma}^{\mu-1}
 +
 C\eps\jb{\sigma}^{-\mu}
 \int_{t_{0,\sigma}}^{t} {\tau}^{2\mu-2} d\tau
 \nonumber\\
 &\le 
 C\eps \left(\jb{\sigma}^{\mu-1}
 +
 \jb{\sigma}^{-\mu} t_{0,\sigma}^{2\mu-1}
 \right)
 \nonumber\\
 &\le 
 C\eps \jb{\sigma}^{\mu-1}
\label{est_V}
\end{align}
for $t\ge t_{0,\sigma}$, 
where $C$ is  independent of $\eps$, $\sigma$, and $\omega$. 

\eqref{est_V} implies 
$$
 |U(t,x)|=|V(t; |x|-t,x/|x|)|\le C\eps\jb{t-|x|}^{\mu-1},
\quad (t,x)\in \Lambda_{T,R}.
$$
Finally, in view of \eqref{est_pa_u_02} with $l=0$, we obtain
\begin{align}
  \sup_{(t,x)\in \Lambda_{T,R}} \jb{t+|x|}\jb{t-|x|}^{1-\mu} |\pa u(t,x)|
  \le C\eps.
  \label{est_after_02}
\end{align}
We remark that the derivation of \eqref{star02} is the only point
where we make use of the condition (H) throughout this proof.
\medskip

\noindent{\bf Step 4: 
Estimates for $|\pa u(t,x)|_{k}$ near the light cone.}

We assume $(t,x)\in \Lambda_{T, R}$ also in this step.
For a non-negative integer $s$, we set 
$$
{\mathcal U}^{(s)}(t,x):=\sum_{|\beta|\le s}|U^{(\beta)}(t,x)|.
$$
Let $1\le |\alpha|\le k$. 
By \eqref{est_pa_u_02} we get
\begin{align}
 |\pa u(t,x)|_{|\alpha|-1}
 \le 
 C\left(t^{-1}{\mathcal U}^{(|\alpha|-1)}(t,x)+\eps t^{\mu-2}\right).
 \label{est_pa_u_03}
\end{align}
It follows from \eqref{rough_bound_u}, \eqref{est_u_sharp01}, 
\eqref{est_pa_u_03}, and Lemma~\ref{lem4_01} that
\begin{align}
|H_\alpha(t,x)|
 \le& 
 C\left(
  \eps^2 t^{2\mu-2}\jb{t-|x|}^{-1}
  + 
  \eps t^{\mu-2}+\eps^2t^{2\mu-3}
  +
  t^{-1} \bigl({\mathcal U}^{(|\alpha|-1)}(t,x)\bigr)^2
 \right)\nonumber\\
 \le& 
 C\eps t^{2\mu-2}\jb{t-|x|}^{-\mu}
 +
 Ct^{-1}\bigl({\mathcal U}^{(|\alpha|-1)}(t,x)\bigr)^2.
 \label{est_H_alpha}
\end{align} 
We put
$$
 V^{(\alpha)}(t; \sigma, \omega)=U^{(\alpha)}\bigl(t, (t+\sigma)\omega\bigr)
$$
for $0 \le t< T$ and $(\sigma,\omega) \in (-\infty, R]\times \Sph^2$.
We fix $(\sigma, \omega)\in (-\infty, R]\times \Sph^2$ and write
$V^{(\alpha)}(t)$ for $V^{(\alpha)}(t; \sigma, \omega)$.
Then \eqref{ode_alpha} is rewritten as 
\begin{align*}
\frac{\pa V^{(\alpha)}}{\pa t}(t)
= -\frac{1}{2t} G_\alpha\left(\omega, V(t), V^{(\alpha)}(t)\right)
+H_\alpha(t, (t+\sigma) \omega)
\end{align*}
for $t_{0,\sigma} < t< T$. Hence by \eqref{est_V} and \eqref{est_H_alpha} 
we obtain
\begin{align*}
 \frac{\pa}{\pa t}|V^{(\alpha)}(t)|^2
 \le& 
 \frac{C}{t}  |V(t)|\,|V^{(\alpha)}(t)|^2
 +2|H_\alpha(t, (t+\sigma) \omega)|\, |V^{(\alpha)}(t)|\\
 \le& 
 \frac{2C^* \eps}{t} |V^{(\alpha)}(t)|^2
 +
 C\left(
  \eps t^{2\mu-2}\jb{\sigma}^{-\mu}
  + 
  t^{-1}\bigl({\mathcal V}^{(|\alpha|-1)}(t)\bigr)^2
 \right)\,|V^{(\alpha)}(t)|,
\end{align*}
where 
$$
 {\mathcal V}^{(s)}(t)\bigl(={\mathcal V}^{(s)}(t;\sigma,\omega)\bigr)
 :=
 \sum_{|\beta|\le s}|V^{(\beta)}(t;\sigma, \omega)|,
$$
and $C^*$ is a positive constant independent of $\alpha$. 
Therefore it follows from \eqref{t_zero} and \eqref{est_U_alpha} that
\begin{align*}
 t^{-C^*\eps}|V^{(\alpha)}(t)|
 \le& 
 t_{0,\sigma}^{-C^*\eps}
 |V^{(\alpha)}(t_{0,\sigma})|
 +
 C\eps \jb{\sigma}^{-\mu} \int_{t_{0,\sigma}}^t \tau^{-C^*\eps+2\mu-2}d\tau
 \\
 &{}+
 C\int_{t_{0,\sigma}}^t \tau^{-C^*\eps-1}
 \bigl({\mathcal V}^{(|\alpha|-1)}(\tau)\bigr)^2 d\tau
 \\
 \le & 
 C\eps\jb{\sigma}^{\mu-1}
 +
 C\int_{2}^t \tau^{-C^*\eps-1}
 \bigl({\mathcal V}^{(|\alpha|-1)}(\tau)\bigr)^2 d\tau.
\end{align*}
To sum up with respect to $|\alpha|\le l$, 
we have
$$
 t^{-C^*\eps}{\mathcal V}^{(l)}(t)
 \le 
  C\eps\jb{\sigma}^{\mu-1}+C\int_{2}^t \tau^{-C^*\eps-1}
  \bigl({\mathcal V}^{(l-1)}(\tau)\bigr)^2 d\tau
$$
for $l \in \{1, \ldots, k\}$. 
Using this inequality, we can show inductively that
\begin{align}
 {\mathcal V}^{(l)}(t)
 \le 
 C \eps \jb{\sigma}^{\mu-1} t^{2^{l-1} C^* \eps} 
\label{est_V_l}
\end{align}
for $t_{0,\sigma}\le t < T$ and $l \in \{1, \ldots, k\}$. 
Indeed, we already know that 
$$
 {\mathcal V}^{(0)}(t)= |V(t)|
\le C\eps\jb{\sigma}^{\mu-1} 
$$
by \eqref{est_V}. Hence we have 
\begin{align*}
 t^{-C^*\eps}{\mathcal V}^{(1)}(t)
 \le  
 C\eps\jb{\sigma}^{\mu-1}
 +
 C\eps^2\jb{\sigma}^{2\mu-2}
 \int_{2}^\infty \tau^{-C^*\eps-1} d\tau
 \le 
 C\eps\jb{\sigma}^{\mu-1},
\end{align*}
which implies \eqref{est_V_l} for $l=1$. 
Next we suppose that \eqref{est_V_l} is true for some 
$l \in \{1, \ldots, k-1\}$. Then we have
\begin{align*}
 t^{-C^*\eps}{\mathcal V}^{(l+1)}(t)
 \le& 
  C\eps\jb{\sigma}^{\mu-1}+C \eps^2\jb{\sigma}^{2\mu-2}
  \int_{2}^t \tau^{(2^{l}-1)C^*\eps-1} d\tau\\
 \le& 
  C\eps \jb{\sigma}^{\mu-1} t^{(2^{l}-1)C^* \eps},
\end{align*}
which yields \eqref{est_V_l} with $l$ replaced by $l+1$. 
Hence \eqref{est_V_l} for $l\in \{1,\ldots,k\}$ 
has been proved.

By \eqref{est_pa_u_02} and \eqref{est_V_l} with $l=k$, we have 
\begin{align*}
|\pa u(t,x)|_k
\le 
C\eps \jb{t+|x|}^{-1+2^{k-1}C^*\eps}\jb{t-|x|}^{-1+\mu},
\quad (t,x)\in \Lambda_{T,R}.
\end{align*}
Finally we take $\eps \le 2^{1-k}\nu/C^*$ to obtain
\begin{align}
 \sup_{(t,x)\in \Lambda_{T, R}} 
 \jb{t+|x|}^{1-\nu}\jb{t-|x|}^{1-\mu}|\pa u(t,x)|_k
  \le C\eps.
 \label{est_after_03}
\end{align}
\medskip
 
\noindent{\bf The final step.}

By \eqref{est_after_01}, \eqref{est_after_02}, and \eqref{est_after_03}, 
we see that there exist two positive constants 
$\eps_2$ and $m$ such that \eqref{est_after} holds for $0<\eps\le \eps_2$. 
This completes the proof of Proposition~\ref{lem_apriori}.
\end{proof}

\section{Asymptotics for the solution to the profile equation} 
\label{ASPE}

This section is devoted to preliminaries for the proof of 
Theorem~\ref{thm_asymp}.
We assume $N=2$ and \eqref{TypicalExample} with $c_0>0$
throughout this section.
Let $u=(u_1, u_2)^{\rm T}$ be the global solution to \eqref{eq}--\eqref{data},
whose existence
is guaranteed by Theorem~\ref{thm_sdge} for small $\eps$,
and let $U=(U_1,U_2)^{\rm T}$ be given by \eqref{U}.
For simplicity of exposition, we introduce a complex-valued function
\begin{equation}
\label{def_U_0}
 \co{U}(t,x):=\sqrt{c_0}U_1(t,x)+ic_0 U_2(t,x),
\end{equation}
where $i=\sqrt{-1}$. Then it follows from \eqref{TypicalExample} and 
\eqref{ode_0} that
\begin{equation}
\label{ode_W}
\pa_+ \co{U}(t,x)
=
-\frac{i}{2t}c(\omega)\left(\realpart \co{U}(t,x)\right)
\co{U}(t,x)+\co{H}(t,x)
\end{equation}
with $\co{H}=\sqrt{c_0}H_1+ic_0 H_2$,
where $c(\omega)$ is given by \eqref{Index}, and $H=(H_1, H_2)^{\rm T}$
by \eqref{DefRemainder}. 

Let $t_0\ge 1$. 
Keeping the application to the profile equation \eqref{ode_W} in mind, 
we consider the following ordinary differential equation for $t> t_0$: 
\begin{align}
 i\frac{dz}{dt}(t)=\frac{\Phi(z(t))}{t}z(t) + J(t), 
 \label{ode_z}
\end{align}
where $\Phi:\C\to \R$  satisfies
\begin{equation}
 |\Phi({z})-\Phi({w})| \le C_0 |{z}-{w}|
 \text{ for } z, w \in \C
\label{assump01}
\end{equation}
with a positive constant $C_0$, and $J:[t_0,\infty)\to \C$ satisfies
\begin{equation}
 |J(t)| \le E_0 t^{-1-\lambda}
\label{assump02}
\end{equation}
with positive constants $E_0$ and $\lambda$. 
The important structure here is that $\Phi$ is real-valued.
Concerning the asymptotics for the solution $z(t)$ of \eqref{ode_z}, we have
the following lemma.
\begin{lem}\label{lem6_001}
Let $z(t)$ be the global solution of \eqref{ode_z}, and suppose 
$$
C_0(E_0t_0^{-\lambda}+|z(t_0)|\lambda)<\lambda^2.
$$
Then there is a $C^1$-function $p=p(s)$ on $[\log t_0,\infty)$ such that
we have 
\begin{align}
 &|z(t)-p(\log t)| 
 \le 
 \frac{E_0\lambda }
 {\{\lambda^2-C_0(E_0t_0^{-\lambda}+|z(t_0)|\lambda)\} t^{\lambda}}, 
 \quad t\ge t_0,
\label{est_z-p}
\end{align}
and 
\begin{equation}
 i\frac{dp}{ds}(s)=\Phi\bigl(p(s)\bigr) p(s),\quad s\ge \log t_0.
\label{ProfEq}
\end{equation}
\end{lem}
To prove Lemma~\ref{lem6_001}, we introduce some sequences.
For the solution $z(t)$ of (\ref{ode_z}), we define sequences 
$\{z_n(t)\}_{n=0}^{\infty}$, $\{\Theta_n(t)\}_{n=0}^\infty$, 
and $\{\zeta_n\}_{n=0}^\infty$ in the following way:
We set $z_0(t)=z(t)$, and inductively define
\begin{align}
\Theta_n(t)
 =& 
 \int_{t_0}^{t} \Phi\bigl(z_{n}(\tau)\bigr) \frac{d\tau}{\tau}, 
 \quad\, t\ge t_0,
\label{def_Theta_n}\\
 \zeta_n
 =& 
 \lim_{\tau\to\infty} z_n(\tau) e^{i\Theta_n(\tau)}, 
\label{def_zeta_n}\\
 z_{n+1}(t)
 =&
 \zeta_n e^{-i\Theta_{n}(t)}, \qquad\qquad t\ge t_0
 \label{def_z_n}%
\end{align}
for $n\in \N_0$, where $\N_0$ denotes the set of non-negative integers. 
In order to see that this definition works well, we have only to check 
the convergence of $\lim_{\tau\to\infty} z_n(\tau)e^{i\Theta_n(\tau)}$ 
for each $n$.

\begin{lem}\label{lem6_002} 
The above sequences $\{z_n(t)\}_{n=0}^{\infty}$, 
$\{\Theta_n(t)\}_{n=0}^\infty$, and $\{\zeta_n\}_{n=0}^\infty$ are well-defined. 
Moreover we have 
\begin{align} 
 \zeta_n=& \left(
  z(t_0) -i\int_{t_0}^{\infty} J(\tau) e^{i\Theta_0(\tau)}d \tau
 \right) \nonumber\\
 & \times \exp\left(
    i\int_{t_0}^{\infty} 
    \left\{\Phi\bigl(z_n(\tau)\bigr)- \Phi\bigl(z_0(\tau)\bigr) \right\}
   \frac{d \tau}{\tau}
 \right)
 \label{first}
\end{align}
and 
\begin{align} 
 |z_{n+1}(t)-z_{n}(t)| \le \frac{E_0}{\lambda t^{\lambda}} 
 \left(\frac{C_0(E_0t_0^{-\lambda}+|z(t_0)|\lambda)}{\lambda^2}\right)^{n}
 \label{second}
\end{align}
for $n \in \N_0$.
\end{lem}

\begin{proof}
We prove Lemma \ref{lem6_002} by the induction on $n$. 

First we consider the case of $n=0$. Since $z_0=z$, it follows from
\eqref{ode_z} that
$$
\left(z_0(t)e^{i\Theta_0(t)}\right)'=-i J(t) e^{i\Theta_0(t)},
$$
which yields
$$
 z_0(t)e^{i\Theta_0(t)}
 =
 z(t_0) -i\int_{t_0}^{t} J(\tau) e^{i\Theta_0(\tau)}d\tau.
$$
This shows that $z_0(\tau)e^{i\Theta_0(\tau)}$ converges as $\tau \to \infty$, 
and that \eqref{first} holds for $n=0$, because \eqref{assump02} implies
$J(\cdot) e^{i\Theta_0(\cdot)} \in L^1(t_0,\infty)$. 
As for (\ref{second}) with $n=0$, we have 
\begin{align*}
 \bigl(z_1(t)-z_0(t)\bigr)e^{i\Theta_0(t)}=
  \zeta_0-z_0(t)e^{i\Theta_0(t)}
 = -i\int_{t}^{\infty} J(\tau) e^{i\Theta_0(\tau)}d\tau,
\end{align*} 
whence
$$
 |z_1(t)-z_0(t)| 
 \le 
 \int_{t}^{\infty} |J(\tau)|d\tau
 \le \frac{E_0}{\lambda t^{\lambda}}.
$$
Note that by \eqref{assump02} we have
\begin{align}
 |\zeta_0|
 =&
 \left|z(t_0)-i\int_{t_0}^\infty J(\tau) e^{i\Theta_0(\tau)}d\tau\right|
 \le 
 |z(t_0)|+\frac{E_0}{\lambda t_0^\lambda}.
 \label{est_zeta_n}
\end{align}

Next we consider the case of $n=n_0+1$ under the assumption that
$\zeta_n$ for $n\le n_0$ are well-defined (thus $z_{n}(t)$ and $\Theta_{n}(t)$ 
for $n\le n_0+1$ are also well-defined), 
and that (\ref{first}) and (\ref{second}) are true for $n\le n_0$. 
We set 
$K=C_0(E_0t_0^{-\lambda}+|z(t_0)|\lambda)/\lambda^2$.
By \eqref{assump01} and \eqref{second} for $n=n_0$, we get
\begin{equation}
 \left| \Phi\bigl(z_{n_0+1}(t)\bigr)-\Phi\bigl(z_{n_0}(t)\bigr) \right| 
 \le 
 C_0 |z_{n_0+1}(t)-z_{n_0}(t)|
 \le 
 \frac{C_0 E_0}{\lambda t^{\lambda}}K^{n_0}.
\label{diff01}
\end{equation}
We put
$$
 \theta_{n_0}=\int_{t_0}^{\infty} 
 \left\{ \Phi\bigl(z_{n_0+1}(\tau)\bigr)-\Phi\bigl(z_{n_0}(\tau)\bigr) \right\}
 \frac{d\tau}{\tau},
$$
which has a finite value because of \eqref{diff01}. 
It also follows from \eqref{diff01} that
\begin{align}
 \bigl| \Theta_{n_0+1}(t)-\Theta_{n_0}(t)-\theta_{n_0} \bigr|
 \le &
 \int_{t}^{\infty} 
 \left| \Phi\bigl(z_{n_0+1}(\tau)\bigr) -\Phi\bigl(z_{n_0}(\tau)\bigr) \right|
 \frac{d\tau}{\tau} \nonumber\\
 \le &
  \frac{C_0 E_0}{\lambda^2 t^{\lambda}}K^{n_0}.
\label{diff02}
\end{align}
Now we obtain from \eqref{def_z_n} for $n=n_0$ and \eqref{diff02} that
\begin{align*}
 \zeta_{n_0+1}
 =&
 \lim_{\tau \to \infty} \bigl( z_{n_0+1}(\tau) e^{i\Theta_{n_0+1}(\tau)} \bigr)
 =  
 \zeta_{n_0} \exp\left(
  i\lim_{\tau\to\infty}\left(\Theta_{n_0+1}(\tau)-\Theta_{n_0}(\tau)\right)
 \right) \\
 = &
 \zeta_{n_0} e^{i\theta_{n_0}},
\end{align*}
which immediately leads to \eqref{first} for $n=n_0+1$ if we replace 
$\zeta_{n_0}$ by the right-hand side of \eqref{first} for $n=n_0$. 
Since $|\zeta_{n_0}|=|\zeta_0|$, it follows from \eqref{def_z_n}, 
\eqref{est_zeta_n}, and \eqref{diff02} that
\begin{align*}
 \bigl| z_{n_0+2}(t)-z_{n_0+1}(t) \bigr|
 &=
 \bigl| 
  \zeta_{n_0} e^{i\theta_{n_0}}e^{-i\Theta_{n_0+1}(t)} 
  -\zeta_{n_0} e^{-i\Theta_{n_0}(t)}
 \bigr|\\
 &\le 
 \bigl| \zeta_{n_0} \bigr| 
 \bigl| \theta_{n_0}-\Theta_{n_0+1}(t)+\Theta_{n_0}(t) \bigr|\\
 &\le
 \left(|z(t_0)|+\frac{E_0}{\lambda t_0^{\lambda}}\right) 
 \frac{C_0 E_0}{\lambda^2 t^{\lambda}} K^{n_0}
 \\
 & =
  \frac{E_0}{\lambda t^{\lambda}} K^{n_0+1},
\end{align*}
which is \eqref{second} for $n=n_0+1$. This completes the proof.
\end{proof}
Now we are in a position to prove Lemma~\ref{lem6_001}.
\begin{proof}[Proof of Lemma~$\ref{lem6_001}$]
Since $z_0$ is continuous on $[t_0,\infty)$, it follows from \eqref{def_Theta_n}
and \eqref{def_z_n} that each $z_n$ is also continuous on $[t_0,\infty)$.
We put $K=C_0(E_0t_0^{-\lambda}+|z(t_0)|\lambda)/\lambda^2$ as before. 
Then we have $0<K<1$ from the assumption. 
By (\ref{second}) we can easily show that $\{z_n(\cdot)\}_{n=0}^\infty$ 
is a uniform Cauchy sequence on $[t_0,\infty)$, and 
$\{z_n(\cdot)\}_{n=0}^\infty$ converges 
uniformly on $[t_0,\infty)$ as $n\to\infty$. 
Hence if we put 
$$
 p(s):=\lim_{n \to \infty}z_n(e^s), \quad s\ge \log t_0,
$$
$p$ is continuous on $[\log t_0,\infty)$.
Since we have $p(\log t)=\lim_{n\to\infty}z_n(t)$ and $0<K<1$, 
it follows from (\ref{second}) that 
\begin{align*}
 |z(t)-p(\log t)|=&\lim_{n\to\infty}|z_0(t)-z_n(t)| \\
 \le &
 \sum_{n=0}^{\infty}|z_{n+1}(t)-z_{n}(t)| 
 \le 
 \sum_{n=0}^{\infty} \frac{E_0}{\lambda t^{\lambda}} K^n  
 \le 
  \frac{E_0}{\lambda (1-K) t^{\lambda}},
\end{align*}
which is \eqref{est_z-p}.
 
To show \eqref{ProfEq}, we set 
$$
 \Theta_{\infty}(t)
 = \int_{t_0}^{t} \Phi\bigl( p(\log \tau)\bigr)\frac{d\tau}{\tau}
 = \int_{\log t_0}^{\log t} \Phi(p(\sigma))d\sigma,
$$
which is well-defined because
the integrands are continuous functions.
Then it follows that 
\begin{align*}
 |\Theta_{\infty}(t)-\Theta_n(t)|
 &\le
 \int_{t_0}^{t} C_0 | p(\log \tau)-z_n(\tau)|\frac{d\tau}{\tau}\\
 &\le 
 \int_{t_0}^{\infty} C_0 \sum_{j=n}^{\infty} 
 \frac{E_0}{\lambda \tau^{\lambda}} K^j \frac{d\tau}{\tau}\\
 &\le 
 \frac{C_0E_0 K^{n}}{\lambda^2(1-K) t_0^{\lambda}}, 
\end{align*}
whence 
${\lim_{n \to \infty}\Theta_n(t)=\Theta_{\infty}(t)}$. 
Similarly we can show
$$
 \lim_{n\to\infty} \int_{t_0}^\infty 
  \left\{\Phi\bigl(z_n(\tau)\bigr)-\Phi\bigl(z_0(\tau)\bigr)\right\} 
 \frac{d\tau}{\tau}
=
 \int_{t_0}^\infty
  \left\{\Phi\bigl(p(\log \tau)\bigr)-\Phi\bigl(z_0(\tau)\bigr)\right\}
 \frac{d\tau}{\tau},
$$
which implies that $\{\zeta_n\}$ converges as $n\to \infty$ with the help of 
\eqref{first} (note that \eqref{est_z-p} shows the existence of the integral 
on the right-hand side of the identity above). 
Thus, by setting ${\zeta_{\infty}=\lim_{n\to \infty} \zeta_n}$, 
we have 
\begin{align*}
 p(s)
 =
 \lim_{n\to\infty}\zeta_{n-1}e^{-i\Theta_{n-1}(e^s)}
 =
 \zeta_{\infty} e^{-i\Theta_{\infty}(e^s)}
 =
 \zeta_{\infty}
 \exp\left(-i\int_{\log t_0}^{s} \Phi\bigl(p(\sigma)\bigr)d\sigma\right). 
\end{align*}
By differentiation, we see that $p(s)$ solves the desired equation 
\eqref{ProfEq}. 
\end{proof}

In the remaining part of this section, we will apply Lemma~\ref{lem6_001} to 
the profile equation \eqref{ode_W}. 
We put 
\begin{equation}
\label{def_V_0}
\co{V}(t;\sigma, \omega)=\co{U}\bigl(t,(t+\sigma)\omega\bigr)
\end{equation}
for $(\sigma, \omega) \in \R \times \Sph^2$ and 
$t>\max\{0,-\sigma\}$. Note that we have
$\co{V}(t;\sigma, \omega)
 =\sqrt{c_0}V_1(t;\sigma, \omega)+ic_0 V_2(t;\sigma, \omega)$,
where $V=(V_1, V_2)^{\rm T}$
is given by \eqref{profile_v}.
Let $R$ be the constant appearing in \eqref{supp_0}. 
It follows from \eqref{ode_W} that $\co{V}(t;\sigma,\omega)$ satisfies 
\begin{equation}
 i\pa_t \co{V}(t;\sigma,\omega)
 = 
 \frac{c(\omega) \realpart\bigl( \co{V}(t;\sigma,\omega) \bigr)}{2t}
 \co{V}(t;\sigma,\omega)
 + 
 i\co{H}\bigl(t, (t+\sigma)\omega\bigr)
\label{eq_i_V}
\end{equation}
for $t>t_{0,\sigma}$ and $\sigma\le R$. 
Note that all the estimates obtained in the proof of 
Proposition~\ref{lem_apriori} 
are valid with $T=\infty$, because we have already shown that \eqref{AE00} 
is valid. On the other hand, for $\sigma>R$, we have 
$$
 \lim_{t\to\infty} \co{V}(t;\sigma,\omega) = \lim_{t\to\infty}0 = 0
$$
because of the finite propagation property \eqref{supp_t}.

As an application of Lemma~\ref{lem6_001}, we have the following. 

\begin{cor} \label{cor6_03}
Let $\eps$ be sufficiently small. 
Suppose that $c(\omega)\not\equiv 0$ on $\Sph^2$.
Then $\lim_{t\to\infty} \co{V}(t;\sigma,\omega)$ exists 
for each $(\sigma, \omega)\in \R\times \Sph^2$. If we put 
$$
 \co{V}^+(\sigma,\omega)
 :=
 \lim_{t \to \infty} \co{V}(t;\sigma,\omega)
$$ 
for each $(\sigma, \omega)\in \R\times \Sph^2$, then we have
\begin{align}
 \realpart \co{V}^+(\sigma,\omega)=0
 \label{dissipation}
\end{align}
for almost all $(\sigma,\omega) \in \R\times \Sph^2$.
Moreover we have $\co{V}^+ \in L^2(\R\times\Sph^2)$ and
\begin{equation}
 \lim_{t\to\infty} \int_{\R\times \Sph^2} 
  |\chi_t(\sigma)\co{V}(t;\sigma,\omega)
   -
   \co{V}^+(\sigma, \omega)|^2 
 d\sigma dS_\omega
 =0,
\label{conv_V_L2}
\end{equation}
where $\chi_t(\sigma)=1$ for $\sigma>-t$, 
and $\chi_t(\sigma)=0$ for $\sigma \le -t$.
\end{cor}

\begin{proof} 
First we show the convergence of $\co{V}(t;\sigma,\omega)$ as 
$t\to\infty$, and \eqref{dissipation}. 
We have only to consider the case $\sigma\le R$, because 
the opposite case is trivial. 
By \eqref{est_H} and \eqref{est_V_ini}, we can apply Lemma~\ref{lem6_001} to 
\eqref{eq_i_V} with $z(t)=\co{V}(t;\sigma,\omega)$, 
$\Phi(z)=c(\omega)(\realpart z)/2$,
$J(t)=i\co{H}(t,(t+\sigma)\omega)$, and $t_0=t_{0,\sigma}$,
provided that $\eps$ is small enough, because we have
$$
 C_0(E_0t_0^{-\lambda}+|z(t_0)|\lambda)\le C_1 \eps < \lambda^2
$$
for $0<\eps<\lambda^2/C_1$, where we have taken $C_0=\max_{\omega\in \Sph^2} c(\omega)/2$, 
$E_0=C\eps\jb{\sigma}^{-\mu}$, and $\lambda=1-2\mu$, while $C_1$ 
is an appropriate positive constant independent of $\sigma$ and $\omega$.
It follows from Lemma \ref{lem6_001} that 
for any $(\sigma,\omega)\in (-\infty, R]\times \Sph^2$, there is 
$p(s)$ satisfying
$$
 i\frac{dp}{ds}(s)=
 \frac{c(\omega) \realpart \bigl( p(s) \bigr)}{2}p(s)
$$
and
$$
 \lim_{t \to \infty}|\co{V}(t;\sigma,\omega)-p(\log t)| =0.
$$
So it is enough to show that $p(s)$ converges as $s\to \infty$, and that 
$\realpart p(s)\to 0$ as $s \to \infty$ for almost all $(\sigma, \omega)\in (-\infty, R]\times \Sph^2$. 
If $c(\omega)=0$, then $p(s)$ is independent of $s$ and the convergence of $p(s)$ as $s\to\infty$ is trivial. 
Since $c(\omega)$ is a polynomial of degree $2$ in $\omega$, the set
of $(\sigma, \omega)\in \R\times \Sph^2$ with $c(\omega)=0$ has measure zero
unless $c(\omega)$ vanishes identically on $\Sph^2$.
Hence we may assume $c(\omega)\ne 0$ from now on, 
and we are going to show 
that $p(s)$ converges to a pure imaginary number as $s\to\infty$.
For this purpose, we set $X(s)=\realpart p(s)/2$, $Y(s)=\imagpart p(s)/2$ 
to rewrite the above equation as 
\begin{align}
 \frac{dX}{ds}(s)= {c}(\omega) X(s)Y(s), \quad \frac{dY}{ds}(s)= -{c}(\omega)X(s)^2.
\label{ode_XY}
\end{align}
We observe that 
$$
 \frac{d}{ds}\Bigl(X(s)^2+Y(s)^2 \Bigr)=0,
$$
which implies that $X(s)^2+Y(s)^2$ is independent of $s$. 
We denote this conserved quantity by $\rho^2$, where $\rho \ge0$. 
The case $\rho=0$ is trivial, because we have $X(s)=Y(s)\equiv 0$. 
Hence we assume $\rho>0$ from now on. 
From the second equation of \eqref{ode_XY} we have 
$$
 \frac{dY}{ds}(s)= c(\omega)\left(Y(s)^2-\rho^2\right).
$$
This can be explicitly integrated as 
$$
 Y(s)=
 \rho\frac{(\rho+\eta)e^{-c(\omega)\rho s}-(\rho -\eta)e^{c(\omega) \rho s}}
 {(\rho+\eta)e^{-c(\omega) \rho s}+(\rho-\eta)e^{c(\omega) \rho s}}
$$
with some real constant $\eta$ satisfying $|\eta|\le \rho$. 
We can also see that 
$$
 X(s)
 =
 \frac{2\rho \xi}
      {(\rho+\eta)e^{-c(\omega)\rho s}+(\rho-\eta)e^{c(\omega)\rho s}}
$$
with some real constant $\xi$ satisfying $\xi^2+\eta^2=\rho^2$. 
If $\xi=0$, then we have $X(s)\equiv 0$, and $Y(s)\equiv \pm \rho$.
If $\xi\ne 0$, then
$\eta^2<\rho^2$. Especially we have $\rho\pm\eta\ne 0$, and we get
\begin{align*}
  \lim_{s \to \infty}X(s)
  = &
  \lim_{s \to \infty}
  \frac{2\rho \xi e^{-|c(\omega)|\rho s}}
       {(\rho \pm \eta)e^{-2|c(\omega)|\rho s}+(\rho\mp \eta)}
  = 
  0, \\
  \lim_{s \to \infty}Y(s)
  = & 
  \rho \lim_{s \to \infty}
  \frac{ \pm\left((\rho \pm \eta)e^{-2|c(\omega)|\rho s}-(\rho\mp \eta)\right) }       {(\rho\pm \eta)e^{-2|c(\omega)|\rho s}+(\rho\mp \eta)}
  = \mp\rho.\,
\end{align*}
where the double sign depends on the signature of $c(\omega)$.
Now the existence of $\lim_{t\to\infty}\co{V}(t;\sigma,\omega)$ and 
\eqref{dissipation} have been established.

It follows from \eqref{AE} and \eqref{rough_bound_u} that
$$
 |\co{U}(t,r\omega)|\le C|U(t,r\omega)|
 =
 C\left|D_-\bigl(ru(t,r\omega)\bigr)\right|
 \le 
 C\eps \jb{t-r}^{-1+\mu}
$$
for any $(t, r, \omega)\in [0,\infty)\times (0,\infty)\times \Sph^2$.
Since $\co{V}(t;\sigma,\omega)
=\co{U}\bigl(t, (t+\sigma)\omega\bigr)$, 
we obtain
\begin{equation}
 |\co{V}(t;\sigma,\omega)|\le C\eps \jb{\sigma}^{-1+\mu}
\label{Est_V_0}
\end{equation}
for $(\sigma,\omega) \in \R\times \Sph^2$ and $t>\max\{0, -\sigma\}$. 
Hence, by taking the limit 
of this inequality as $t \to \infty$, we have 
$$
 |\co{V}^+(\sigma,\omega)|\le C\eps \jb{\sigma}^{-1+\mu},
 \quad (\sigma,\omega)\in \R\times \Sph^2,
$$
which shows $\co{V}^+\in L^2(\R\times\Sph^2)$ since $\mu<1/2$. 
Furthermore we have
$$
 |\chi_t(\sigma) \co{V}(t;\sigma,\omega)
  -
  \co{V}^+(\sigma, \omega)|^2
  \le 
  C\eps^2 \jb{\sigma}^{-2+2\mu}
  \in L^1(\R\times \Sph^2)
$$
for $t\ge 0$. 
Now, since 
$\lim_{t\to\infty}|\chi_t(\sigma)\co{V}(t;\sigma,\omega)
-\co{V}^+(\sigma, \omega)|^2=0$ 
for each $(\sigma,\omega)\in \R\times \Sph^2$, 
Lebesgue's convergence theorem implies \eqref{conv_V_L2}. 
This completes the proof.
\end{proof}

\section{Proof of Theorem~\ref{thm_asymp}} 
\label{PT2}
In the following, we write 
$$
\hat{\omega}(x)
=
\bigl(\hat{\omega}_a(x)\bigr)_{a=0,1,2,3}=(-1, x_1/|x| , x_2/|x|, x_3/|x|)
$$ 
for $x\in \R^3\setminus\{0\}$.
For the proof of Theorem~\ref{thm_asymp}, we will use the following lemma:
\begin{lem}\label{KatLem}
Let $\phi\in C\bigl([0,\infty); \dot{H}^1(\R^3)\bigr)\cap 
C^1\bigl([0,\infty); L^2(\R^3)\bigr)$. 
The following assertions {\rm (i)} and {\rm (ii)} are equivalent: \\
{\rm (i)} 
There exists 
$(\phi_0^+,\phi_1^+)\in \calH(\R^3)=\dot{H}^1(\R^3)\times L^2(\R^3)$ 
such that 
$$
\lim_{t\to\infty} \|\phi(t)-\phi^+(t)\|_E=0,
$$
where $\phi^+\in C\bigl([0,\infty); \dot{H}^1(\R^3)\bigr)\cap 
C^1\bigl([0,\infty); L^2(\R^3)\bigr)$ 
is the unique solution to 
$\dal\phi^+=0$ with $(\phi^+, \pa_t \phi^+)(0)=(\phi_0^+, \phi_1^+)$.\\
{\rm (ii)} 
There is a function 
$P=P(\sigma, \omega)\in L^2(\R\times \Sph^2)$ such that
$$
\lim_{t\to\infty} 
\|\pa \phi(t,\cdot)-\hat{\omega}(\cdot)P^\sharp(t,\cdot)\|_{L^2(\R^3)}=0,
$$
where $P^\sharp$ is given by
$$
P^\sharp (t,x)=\frac{1}{|x|}P(|x|-t, |x|^{-1}x),\quad x\ne 0.
$$
\end{lem}
See \cite{ka2} for the proof (see also \cite{ka}, where the above result was 
implicitly proved). 
We note that $(\phi_0^+, \phi_1^+)$ and $P$ above 
are related by $P={\mathcal T}[\phi_0^+,\phi_1^+]$,
where ${\mathcal T}[\phi_0^+,\phi_1^+]$ is the 
so-called translation representation of $(\phi_0^+,\phi_1^+)$ 
introduced by Lax-Phillips~\cite[Chapter IV]{lp}. 
More precisely, ${\mathcal T}$ is an isometric isomorphism from 
$\calH(\R^3)$ to $L^2(\R\times \Sph^2)$ 
which can be represented as
$$
{\mathcal T}[\phi_0,\phi_1](\sigma,\omega)=\frac{1}{4\pi}\left(-\pa_\sigma{\mathcal R}[\phi_0](\sigma, \omega)+{\mathcal R}[\phi_1](\sigma, \omega)\right), \quad
(\sigma, \omega)\in \R\times \Sph^2,
$$
for $(\phi_0,\phi_1)\in C^\infty_0(\R^3)\times C^\infty_0(\R^3)$, 
where ${\mathcal R}[\psi]$ is the Radon transform of $\psi$, given by
$$
{\mathcal R}[\psi](\sigma, \omega)=\int_{y\cdot \omega=\sigma} \psi(y) dS_y
$$
with the surface element $dS_y$ on the plane $\{y\in \R^3;y\cdot\omega=\sigma\}$.

\begin{proof}[Proof of Theorem~$\ref{thm_asymp}$]
Let $u=(u_1,u_2)^{\rm T}$ be the global solution to \eqref{eq}--\eqref{data}
with \eqref{TypicalExample} for small $\eps$, and $U=(U_1, U_2)^{\rm T}$ and 
$V=(V_1, V_2)^{\rm T}$ be
given by \eqref{U} and \eqref{profile_v}, respectively.
Suppose that $c_0>0$ and $c(\omega)\not\equiv 0$ on $\Sph^2$.
Recall that all the estimates in the proof of Proposition~\ref{lem_apriori} 
are valid in our present setting. 

As in the previous section, we define 
$\co{U}=\co{U}(t,x)$ by \eqref{def_U_0},
and $\co{V}=\co{V}(t;\sigma, \omega)$ by \eqref{def_V_0}.
We write 
$\co{V}^+(\sigma, \omega)
 =
 \lim_{t\to \infty} \co{V}(t;\sigma, \omega)$ 
whose existence is guaranteed by Corollary~\ref{cor6_03}.
If we put
\begin{align*}
V_1^+(\sigma, \omega)= c_0^{-1/2} \realpart \co{V}^+(\sigma, \omega)
\text{ and }
V_2^+(\sigma, \omega)= & c_0^{-1} \imagpart \co{V}^+(\sigma, \omega),
\end{align*}
then Corollary~\ref{cor6_03} implies that $V_1^+(\sigma, \omega)=0$ almost 
everywhere and $V_2^+\in L^2(\R\times \Sph^{2})$. Hence, 
if we can prove
\begin{equation}
\label{Goal01}
\lim_{t\to \infty} 
\sum_{j=1}^2 
 \|
   \pa u_j(t,\cdot)-\widehat{\omega}(\cdot)V_j^{+, \sharp}(t,\cdot)
 \|_{L^2(\R^3)}
 =
 0,
\end{equation}
then we obtain \eqref{asymp_u1} immediately, and also \eqref{asymp_u2}
with the help of Lemma~\ref{KatLem}, 
where
$$
V_j^{+,\sharp}(t,x):=\frac{1}{|x|} V_j^+(|x|-t, |x|^{-1}x),\quad x\ne 0
$$
for $j=1,2$.
We define
\begin{align*}
J_1(t)
= &
\left(\sum_{j=1}^2 \int_{\Sph^2}\left(\int_0^\infty 
  |r \pa u_j(t, r\omega)-\widehat{\omega}(r\omega)V_j(t;r-t,\omega)|^2 
 dr \right)dS_\omega\right)^{1/2},\\
J_2(t)
= &
\left(\sum_{j=1}^2 \int_{\Sph^2}\left(\int_0^\infty 
  |\widehat{\omega}(r\omega) V_j(t;r-t,\omega)
   -
   r\widehat{\omega}(r\omega)V_j^{+, \sharp}(t, r\omega)|^2 
 dr\right)dS_\omega\right)^{1/2}.
\end{align*}
It follows from
\eqref{est_pa_t}, \eqref{est_pa_j}, and  \eqref{rough_bound_u} that
\begin{align*}
 J_1(t)^2 
 \le & 
  C \int_{\Sph^2}\left(\int_0^\infty |u(t,r\omega)|_1^2 dr\right)dS_\omega 
 \le 
  C\eps^2\int_0^\infty \jb{t+r}^{2\mu-2}dr\\
 \le & 
  C \eps^2\jb{t}^{2\mu-1}\to 0
\end{align*}
as $t\to\infty$.
Therefore \eqref{Goal01} follows from 
\begin{equation}
\label{Goal02}
\lim_{t\to \infty} J_2(t)=0,
\end{equation}
because we have $\sum_{j=1}^2\|\pa u_j(t)-\widehat{\omega}V_j^{+,\sharp}(t)\|_{L^2}\le J_1(t)+J_2(t)$.
In order to prove \eqref{Goal02},
we introduce
$$
 \co{V}^{+, \sharp}(t,x)
 =
 \frac{1}{|x|}\co{V}^+(|x|-t, |x|^{-1}x),\quad x\ne 0.
$$
Let
$$
J_3(t)
= 
\left(\int_{\Sph^2}\left(\int_0^\infty 
  |\widehat{\omega}(r\omega) \co{V}(t;r-t,\omega)
   -
   r\widehat{\omega}(r\omega) \co{V}^{+, \sharp}(t, r\omega)|^2 
 dr\right)dS_\omega\right)^{1/2}.
$$
By \eqref{conv_V_L2} we get
\begin{align*}
 J_3(t)^2
 =& 
 2 \int_{\Sph^2}\left( \int_0^\infty 
    |\co{V}(t;r-t,\omega)-\co{V}^{+}(r-t, \omega)|^2 
   dr\right)dS_\omega
  \nonumber\\
 =& 2 \int_{\Sph^2}\left( \int_{-t}^\infty 
    |\chi_t(\sigma) \co{V}(t;\sigma,\omega)
     -
     \co{V}^{+}(\sigma, \omega)|^2 
   d\sigma\right)dS_\omega
  \nonumber\\
 \le & 
  2 \int_{\Sph^2}\left(\int_\R 
    |\chi_t(\sigma) \co{V}(t;\sigma,\omega)
     -
     \co{V}^{+}(\sigma, \omega)|^2 
  d\sigma\right)dS_\omega
\to 0
\end{align*}
as $t\to\infty$, because $\chi_t(\sigma)=1$ for $\sigma>-t$.
Since $J_2(t)\le CJ_3(t)$, we obtain \eqref{Goal02} immediately.

It remains to prove \eqref{AsympProfile}.
We set 
$$
\|u(t)\|_{\widetilde{E}}^2:=c_0^{-1} \|u_1(t)\|_E^2+\|u_2(t)\|_E^2
$$
for $u=(u_1,u_2)^{\rm T}$. 
By the standard argument of the energy, we have 
\begin{align*}
\frac{d}{dt}\left(\|u(t)\|_{\widetilde{E}}^2\right)
=& \int_{\R^3} \frac{F_1\bigl(\pa u(t,x)\bigr)}{c_0} \pa_t u_1(t,x) dx
{}+\int_{\R^3} F_2\bigl(\pa u(t,x)\bigr) \pa_t u_2(t,x) dx.
\end{align*}
Let $R$ be the constant appearing in \eqref{supp_0}, and we put
$\Lambda_{\infty,R}=\{(t,x)\in [0,\infty)\times \R^3; 1\le t/2\le |x|\le t+R\}$.
We put $\chi(t,x)=1$ if $(t,x)\in \Lambda_{\infty,R}$, and $\chi(t,x)=0$ 
otherwise. 
Since $\bigl(1-\chi(t,x)\bigr)|\pa u(t,x)|\le C\eps \jb{t+r}^{\mu-2}$ 
by \eqref{AE}, it follows from \eqref{Ene03} that
\begin{align*}
\sum_{j=1}^2\int_{\R^3}\bigl(1-\chi(t,x)\bigr)\left|F_j\bigl(\pa u(t,x)\bigr)(\pa_t u_j)(t,x)\right|dx\le & C\eps (1+t)^{\mu-2}\|\pa u(t)\|_{L^2}^2\\
\le & C\eps^2(1+t)^{(3\mu/2)-2}\|u(t)\|_{\widetilde{E}}
\end{align*}
for sufficiently small $\eps$.
For $(t,x)\in \Lambda_{\infty, R}$, we obtain from
\eqref{est_pa_t} and \eqref{est_pa_j} that
$$
|\pa_a u(t,x)-\omega_a D_- u(t,x)|\le C\jb{t+r}^{-1} |u(t,x)|_1,
$$
which leads to
\begin{align*}
|(\pa_a u_k)(\pa_b u_l)(\pa_t u_j)+\omega_a\omega_b(D_-u_k)(D_-u_l)(D_-u_j)|
\le & C \jb{t+r}^{-1}|u|_1|\pa u|^2\\
\le & C\eps(1+t)^{\mu-2}|\pa u|^2
\end{align*}
with the help of \eqref{rough_bound_u}.
As an immediate consequence, we obtain
\begin{align*}
\frac{F_1(\pa u)}{c_0}(\pa_tu_1)+F_2(\pa u)(\pa_t u_2)
=& -\frac{F_1^{\rm red}(\omega, D_-u)}{c_0}(D_-u_1)\\
& -F_2^{\rm red}(\omega, D_-u)(D_-u_2)
+O(\eps(1+t)^{\mu-2}|\pa u|^2)\\
=& O(\eps(1+t)^{\mu-2}|\pa u|^2)
\end{align*}
for $(t,x)\in \Lambda_{T,R}$, because of the structure \eqref{TypicalExample}.
Therefore we get
\begin{align*}
\int_{\R^3} \chi(t,x)\left|\frac{F_1(\pa u)}{c_0}(\pa_tu_1)+F_2(\pa u)(\pa_t u_2)\right|dx
\le & C \eps(1+t)^{\mu-2}\|\pa u(t)\|_{L^2}^2\\
\le & C\eps^2 (1+t)^{(3\mu/2)-2}\|u(t)\|_{\widetilde{E}},
\end{align*}
provided that $\eps$ is small enough.
To sum up, we obtain
$$
\left|\frac{d}{dt}\left(\|u(t)\|_{\widetilde{E}}^2\right)\right|
\le C\eps^2 (1+t)^{(3\mu/2)-2}\|u(t)\|_{\widetilde{E}},
$$
which yields
\begin{align*}
\left| \|u(t)\|_{\widetilde{E}}-\|u(0)\|_{\widetilde{E}}\right|
\le & C\eps^2
\int_0^\infty (1+\tau)^{(3\mu/2)-2}d\tau \le C\eps^2.
\end{align*}
Since we have $\|u_2^+(t)\|_E=\|u_2^+(0)\|_E$,
it follows that
\begin{align*}
\left|\|u(0)\|_{\widetilde{E}}-\|u_2^+(0)\|_E\right|
\le & \left|\|u(0)\|_{\widetilde{E}}-\|u(t)\|_{\widetilde{E}}\right|
+\left|\|u(t)\|_{\widetilde{E}}-\|u_2^+(t)\|_E\right|\\
\le & C\left(\eps^2+\|u_1(t)\|_E+\|u_2(t)-u_2^+(t)\|_E\right).
\end{align*}
By \eqref{asymp_u1} and \eqref{asymp_u2}, taking the limit as $t\to\infty$ in
the inequality above, we obtain
$$
\left|\|u(0)\|_{\widetilde{E}}-\|u_2^+(0)\|_E\right|\le
C \eps^2,
$$ 
which immediately yields \eqref{AsympProfile}. This completes the proof.
\end{proof}
\section{Asymptotic behavior for general two-component systems 
under the condition (H)}
\label{ConclRem}

In this section, we discuss the asymptotic behavior for general 
two-component systems which are not necessarily of the form 
\eqref{TypicalExample}. 
If the condition (H) is satisfied with some ${\mathcal A}(\omega)$,
then the condition (H) with ${\mathcal A}(\omega)$ replaced by $h(\omega){\mathcal A}(\omega)$ remains valid for an arbitrary continuous function $h$ on $\Sph^2$ with positive values. Therefore, without loss of generality, we may assume that ${\mathcal A}(\omega)$ 
has $1$ and $c_0(\omega)$ as its eigenvalues, 
where $c_0$ is a positive and continuous function on $\Sph^2$. Then
we can take an orthogonal matrix ${\mathcal P}(\omega)$ 
such that
$$
{\mathcal A}(\omega)={\mathcal P}(\omega)^{\rm T} \left(\begin{matrix} 1 & 0 \\
0 & c_0(\omega) \end{matrix} \right){\mathcal P}(\omega).
$$
Since the condition (H) yields 
$$
\left({\mathcal P(\omega)}Y\right)^{\rm T} \left(\begin{matrix} 1 & 0 \\
0 & c_0(\omega) \end{matrix} \right) {\mathcal P}(\omega) F^{\rm red}(\omega, Y)=0,
$$
we see that 
${\mathcal P}(\omega) 
 F^{\rm red}\left(\omega, {\mathcal P}(\omega)^{\rm T}\widetilde{Y}\right)$ 
is perpendicular to $( \widetilde{Y}_1, c_0(\omega)\widetilde{Y}_2)^{\rm T}$
for all $\widetilde{Y}=\bigl(\widetilde{Y}_1, \widetilde{Y}_2\bigr)^{\rm T}\in \R^2$, 
by substituting $Y={\mathcal P}(\omega)^{\rm T} \widetilde{Y}$. 
Accordingly, we deduce that 
\begin{equation}
\label{diagonalize}
 {\mathcal P}(\omega) 
 F^{\rm red}\left(\omega, {\mathcal P}(\omega)^{\rm T}\widetilde{Y}\right)
 =
 \left(\widetilde{c}_1(\omega)\widetilde{Y}_1+\widetilde{c}_2(\omega)\widetilde{Y}_2\right)
 \left(\begin{matrix}
   -c_0(\omega)\widetilde{Y}_2\\
   \widetilde{Y}_1
  \end{matrix}\right)
\end{equation}
with some $\widetilde{c}_1(\omega)$ and $\widetilde{c}_2(\omega)$. 
Here $\widetilde{c}_1$ and $\widetilde{c}_2$ are bounded functions on 
$\Sph^2$. In fact,
substituting $\widetilde{Y}=(1, 0)^{\rm T}$ in \eqref{diagonalize}, 
we find that
$$
|\widetilde{c}_1(\omega)|=\left|{\mathcal P}(\omega)F^{\rm red}\left(\omega, {\mathcal P}(\omega)^{\rm T}\widetilde{Y}\right)\right|\le \max_{\eta\in \Sph^2, |Y|=1}|F^{\rm red}(\eta, Y)|,
\quad \omega\in \Sph^2,
$$
and a similar estimate for $\widetilde{c}_2$ can be obtained by 
choosing $\widetilde{Y}=(0,1)^{\rm T}$. 
It is easy to see that the null condition is satisfied if and only if 
$\widetilde{c}_1(\omega)^2+\widetilde{c}_2(\omega)^2=0$ for all $\omega\in \Sph^2$. 
Moreover, the set of $\omega\in \Sph^2$ with 
$\widetilde{c}_1(\omega)^2+\widetilde{c}_2(\omega)^2=0$ is of surface measure zero when
the null condition is violated. Indeed,
for $\omega$ satisfying $\widetilde{c}_1(\omega)^2+\widetilde{c}_2(\omega)^2=0$, we find 
$F^{\rm red}(\omega, {\mathcal P}(\omega)^{\rm T}\widetilde{Y})=0$ 
for all $\widetilde{Y}\in \R^2$, and hence
$F^{\rm red}(\omega,Y)=0$ for all $Y\in \R^2$; if the null condition is 
violated, then the set of such $\omega$ has surface measure zero, 
since the coefficients of $Y_kY_l$ with $k,l\in\{1,2\}$ in $F^{\rm red}(\omega, Y)$
are polynomials of degree $2$ in $\omega$.

Suppose that the condition (H) is satisfied, but the null condition is violated.
Let $u=(u_1, u_2)^{\rm T}$ be the global solution to \eqref{eq}--\eqref{data}. 
Let $U=(U_1, U_2)^{\rm T}$ and $V=(V_1, V_2)^{\rm T}$ be given by \eqref{U} 
and \eqref{profile_v}, respectively. We put 
$$
\widetilde{V}(t;\sigma, \omega)
=
\bigl(
 \widetilde{V}_1(t;\sigma, \omega), \widetilde{V}_2(t;\sigma, \omega)
\bigr)^{\rm T}
=
{\mathcal P}(\omega)V(t;\sigma, \omega),
$$
and
\begin{align*}
\co{\widetilde{V}}(t;\sigma, \omega)
=& 
 \sqrt{c_0(\omega)}\widetilde{V}_1(t;\sigma,\omega)
 +
 ic_0(\omega)\widetilde{V}_2(t;\sigma, \omega)
=
 {\mathcal C}(\omega)^{\rm T}\widetilde{V}(t;\sigma, \omega)
\end{align*}
with 
${\mathcal C}(\omega)=
\left(\sqrt{c_0(\omega)},\, i c_0(\omega)\right)^{\rm T}$. 
Multiplying \eqref{ode_V} by 
${\mathcal C}(\omega)^{\rm T}{\mathcal P}(\omega)$ from 
the left, and using \eqref{diagonalize}, we get
\begin{align}
\pa_t \co{\widetilde{V}}(t)
=&
-\frac{1}{2t} {\mathcal C}(\omega)^{\rm T}{\mathcal P}(\omega)
F^{\rm red}\left(\omega, {\mathcal P}(\omega)^{\rm T}\widetilde{V}(t)\right)
+ 
\co{\widetilde{H}}\bigl(t, (t+\sigma)\omega\bigr)
\nonumber\\
=& 
-\frac{i}{t}\Psi\bigl(\co{\widetilde{V}}(t)\bigr)
 \co{\widetilde{V}}(t)
+
\co{\widetilde{H}}\bigl(t, (t+\sigma)\omega\bigr),
\end{align}
where 
$\co{\widetilde{H}}\bigl(t,x)
={\mathcal C}(|x|^{-1}x)^{\rm T}{\mathcal P}(|x|^{-1}x)H(t,x)$ 
and
$$
 \Psi(z)
 = 
 \frac{1}{2}\left(
 \widetilde{c}_1(\omega)\left(\realpart z\right)
 +
 \frac{\widetilde{c}_2(\omega)}{\sqrt{c_0(\omega)}}
 \left(\imagpart z \right) \right).
$$
In view of Lemma~\ref{lem6_001}, we need to solve
\begin{equation}
\label{ProfileF}
i \frac{d \widetilde{p}}{ds}(s)
= 
\Psi\bigl(\widetilde{p}(s)\bigr) \widetilde{p}(s)
\end{equation}
in order to specify the asymptotic profile of 
$\co{\widetilde{V}}(t;\sigma,\omega)$ for fixed $(\sigma, \omega)$. 
As is done in Section~\ref{ASPE}, we introduce  
$$
\left( \begin{matrix}
\widetilde{X}(s)\\
\widetilde{Y}(s)
\end{matrix}
\right)=
\frac{1}{2\sqrt{c_0(\omega)}}\left(
\begin{matrix}
\sqrt{c_0(\omega)}\widetilde{c}_1(\omega) & \widetilde{c}_2(\omega)\\
\widetilde{c}_2(\omega) & -\sqrt{c_0(\omega)}\widetilde{c}_1(\omega)
\end{matrix}
\right)
\left(
\begin{matrix}
\realpart \widetilde{p}(s)\\
\imagpart \widetilde{p}(s)
\end{matrix}
\right)
$$
so that we can reduce \eqref{ProfileF} to the simpler system
$$\
 \frac{d \widetilde{X}}{ds}(s)=-\widetilde{X}(s)\widetilde{Y}(s),
 \quad 
 \frac{d \widetilde{Y}}{ds}(s)=\widetilde{X}(s)^2.
$$ 
Now going similar lines to the proof of Corollary~\ref{cor6_03} we see that
\begin{equation}
\lim_{t\to\infty} \co{\widetilde{V}}(t;\sigma, \omega)
=\co{\widetilde{V}}^+(\sigma,\omega)
\end{equation}
for almost every $(\sigma, \omega)\in \R\times \Sph^2$,
where
$$
\co{\widetilde{V}}^+(\sigma,\omega)
=\frac{2\sqrt{c_0(\omega)}\left(\widetilde{c}_2(\omega)-i\sqrt{c_0(\omega)}\widetilde{c}_1(\omega)\right)}{c_0(\omega)\widetilde{c}_1(\omega)^2+\widetilde{c}_2(\omega)^2} \rho(\sigma, \omega)
$$
with some function $\rho=\rho(\sigma, \omega)$ when 
$\widetilde{c}_1(\omega)^2+\widetilde{c}_2(\omega)^2\ne 0$. 
Since we have 
$|\co{\widetilde{V}}(t;\sigma, \omega)|\le C\eps \jb{\sigma}^{-1+\mu}$ 
as in \eqref{Est_V_0}, we can prove 
$\co{\widetilde{V}}^+\in L^2(\R\times \Sph^2)$ 
and
\begin{equation}
 \lim_{t\to\infty} \int_{\R\times \Sph^2} 
 \left|\chi_t(\sigma)\co{\widetilde{V}}(t;\sigma, \omega)
       -
       \co{\widetilde{V}}^+(\sigma, \omega)\right|^2 d\sigma dS_\omega
 =
 0
\end{equation}
as before. Now, we put
\begin{align}
\widetilde{V}^+(\sigma, \omega)=& \left(
\begin{matrix}
\widetilde{V}_1^+(\sigma, \omega) \\
\widetilde{V}_2^+(\sigma, \omega)
\end{matrix}
\right)
=\frac{1}{c_0(\omega)} \left(
\begin{matrix}
\sqrt{c_0(\omega)} \realpart \co{\widetilde{V}}^+(\sigma, \omega)\\
\imagpart \co{\widetilde{V}}^+(\sigma, \omega)
\end{matrix}
\right)
\nonumber\\
=&\frac{2\rho(\sigma,\omega)}{c_0(\omega)\widetilde{c}_1(\omega)^2+\widetilde{c}_2(\omega)^2}
\left(
\begin{matrix}
\widetilde{c}_2(\omega)\\
-\widetilde{c}_1(\omega)
\end{matrix}
\right), 
\nonumber\\
V^+(\sigma, \omega)=&
\left(
\begin{matrix}
V_1^+(\sigma, \omega) \\
V_2^+(\sigma, \omega)
\end{matrix}
\right)
={\mathcal P(\omega)}^{T}\widetilde{V}^+(\sigma, \omega)
\nonumber\\
=&
\frac{2\rho(\sigma,\omega)}{c_0(\omega)\widetilde{c}_1(\omega)^2+\widetilde{c}_2(\omega)^2}
{\mathcal P}(\omega)^{\rm T}\left(
\begin{matrix}
\widetilde{c}_2(\omega)\\
-\widetilde{c}_1(\omega)
\end{matrix}
\right).  
\label{GeneralAsymptotics}
\end{align}
Then, recalling that ${\mathcal P}(\omega)$ is an orthogonal matrix, we have
\begin{align*}
& \int_{\Sph^2} 
  \left(\int_0^\infty 
    \left|V(t; r-t, \omega)-r V^{+,\sharp}(t, r\omega)\right|^2 dr
  \right) dS_\omega\\
& \quad =
  \int_{\Sph^2} \left(\int_0^\infty 
   \left|\widetilde{V}(t; r-t, \omega)
         -
         r \widetilde{V}^{+,\sharp}(t,r \omega)
   \right|^2 
  dr\right) dS_\omega\\
& \quad \le 
  C \int_{\R\times \Sph^2}
  \left|
    \chi_t(\sigma)\co{\widetilde{V}}(t;\sigma, \omega)
    -
    \co{\widetilde{V}}^+(\sigma, \omega)
  \right|^2 
  d\sigma dS_\omega\to 0
\end{align*}
as $t\to \infty$, where $V^{+,\sharp}$ and $\widetilde{V}^{+,\sharp}$
are defined from $V^+$ and $\widetilde{V}^+$ as before.
Finally, 
noting that \eqref{GeneralAsymptotics} implies
$c_1(\omega)V_1^+(\sigma, \omega)+c_2(\omega)V^+_2(\sigma, \omega)=0$
with
$$
\left(\begin{matrix} c_2(\omega)\\ -c_1(\omega) \end{matrix}\right)
:={\mathcal P}(\omega)^{\rm T} \left( \begin{matrix} \widetilde{c}_2(\omega)\\ -\widetilde{c}_1(\omega)
\end{matrix}\right),
$$
we can modify the proof of Theorem~\ref{thm_asymp} to obtain 
the following:
\begin{thm}\label{thm_general_asymp}
Suppose that $N=2$ and the condition {\rm (H)} is satisfied,
but the null condition is violated.
Let $\eps$ be sufficiently small, and $u=(u_1, u_2)^{\rm T}$ be
the global solution to \eqref{eq}--\eqref{data}. Then there
is $(f_j^+, g_j^+)\in \calH(\R^3)$
such that 
$$
\lim_{t\to\infty} \|u_j(t)-u_j^+(t)\|_E=0,\quad j=1,2,
$$
where $u_j^+$ is the solution to the free wave equation $\Box u_j^+=0$
with initial data $(u_j^+, \pa_t u_j^+)(0)=(f_j^+, g_j^+)$. Moreover,
there are bounded functions $c_1=c_1(\omega)$ and $c_2=c_2(\omega)$ of 
$\omega\in \Sph^2$ such that $\bigl(c_1(\omega), c_2(\omega)\bigr)\not\equiv (0,0)$ and
\begin{equation}
\label{SR}
 c_1(\omega){\mathcal T}[f_1^+, g_1^+](\sigma,\omega)
 +
 c_2(\omega){\mathcal T}[f_2^+, g_2^+](\sigma, \omega)
 =
 0
\end{equation}
for almost all $(\sigma,\omega)\in \R\times \Sph^2$,
where $\mathcal T$ is the translation representation.
Here $c_1$ and $c_2$ depend only on the coefficients of the nonlinearity $F$.
\end{thm}
\begin{rmk}
The result of Theorem~\ref{thm_asymp} corresponds to the case where 
$c_1(\omega)\equiv 1$ and $c_2(\omega)\equiv 0$ in 
Theorem~\ref{thm_general_asymp}.
\end{rmk}

We conclude this paper with the following remark:
From
Theorem~\ref{thm_general_asymp}, 
we see that the global solution for small data
to a two-component system satisfying the condition (H) and violating the null 
condition is asymptotically free, 
but there is a strong relationship \eqref{SR} between the asymptotic profiles 
for the components $u_1$ and $u_2$. This is the special feature of the 
condition (H) with $N=2$. 
Since the solution for \eqref{ThirdExampleA} is not always 
asymptotically free, 
Theorem~\ref{thm_general_asymp} cannot be extended to the case $N\ge 3$ 
directly; there might be a wider variety of asymptotic behavior.

\section*{Acknowledgments}
The authors would like to express their sincere gratitude to
Professor Akitaka Matsumura for his comments on the earlier version of
this work.
The work of S.~K. 
is supported by Grant-in-Aid for Scientific Research (C) (No.~23540241), 
JSPS. 
The work of H.~S. 
is supported by Grant-in-Aid for Young Scientists~(B) (No.~22740089)
and Grant-in-Aid for Scientific Research (C) (No.~25400161), JSPS.


\end{document}